%% file: Harmonic_maps_with_prescribed_degrees_on_the_boundary_of_an_annulus_and_bifurcation_of_Catenoids.tex
\documentclass[a4report,12pt]{article}
\usepackage[latin1]{inputenc}
\usepackage{amscd}
\usepackage{amsmath,amsthm,amssymb,amsfonts}
\usepackage{mathrsfs}
\usepackage{makeidx}
\usepackage{graphicx}
\usepackage{pstricks-add}

\usepackage{color}
\usepackage{amsmath}
\usepackage{amsfonts}
\usepackage{amssymb}
\usepackage{geometry}
\usepackage{bbold}
\newtheorem{theorem}{Theorem}[section]
\newtheorem{proposition}{Proposition}[section]
\newtheorem{lemma}{Lemma}[section]
\newtheorem{remark}{Remark}[section]
\newtheorem{corollary}{Corollary}[section]
\theoremstyle{Definition}
\newtheorem{definition}{Definition}[section]

\newtheorem{example}{Example}[section]

\def\R{\mathbb{R}}
\def\C{\mathbb C}
\def\E{\mathcal E}
\def \H {\mathcal{H}}
\def \p{\varphi}

\def \e {\varepsilon}
\def \I {\mathcal I}
\def \F {\mathcal F}

\def \n {|n|}
\def \p {|p|}
\def \q {|q|}

\DeclareMathOperator{\capa}{cap}

\title{Harmonic maps with prescribed degrees on the boundary of an annulus and bifurcation of catenoids}
\author{Laurent Hauswirth \and R\'emy Rodiac}
\date{}

\begin{document}

\maketitle

\begin{abstract}
Let $A \subset \mathbb{R} ^2 $ be a smooth doubly connected domain. We consider the Dirichlet energy $E(u)=\int_{A} |\nabla 
u|²$, where $u:A \rightarrow \mathbb{C}$, and look for critical points of this energy with prescribed modulus $|u|=1$ on 
$\partial A$ and with prescribed degrees on the two connected components of $\partial A$. This variational 
problem is a problem with lack of compactness hence we can not use the direct methods of calculus of variations. Our 
analysis relies on the so-called Hopf differential and on a strong link between this problem and the problem of 
finding all minimal surfaces bounded by two $p$ covering of circles in parallel planes. We then construct new immersed minimal 
surfaces in $\mathbb{R} ^3$ with this property. These surfaces are obtained by bifurcation from a family of $p$-coverings of catenoids.
\end{abstract}

\section{Introduction and statement of the results}

Let $A \subset \mathbb{R}^2\simeq \C $ be a smooth, bounded, doubly connected domain of the form $A=\Omega 
\setminus \overline{\omega}$, where $\Omega$ and $\omega$ are simply connected smooth bounded domains and 
$\overline{\omega} \subset \Omega \subset \R^2$. We are interested here in solutions $u:A \rightarrow \C$ of the following equations

\begin{equation}\label{semi-stiff 1}
\left\{
\begin{array}{rcll}
\Delta u & =& 0, \ \ \text{in} \ A, \\
|u|&=&1, \ \ \text{a.e.} \ \text{on} \ \partial  A, \\
u\wedge \partial_\nu u & =& 0, \ \ \text{a.e.} \  \text{on} \ \partial A.
\end{array}
\right.
\end{equation}

\noindent where $\nu$  stands for the outer unit normal to $\partial A$ and $a\wedge b$ stands for the 
determinant of two 
vectors $a,b$ in $\R^2$.
\\

\noindent Solutions of \eqref{semi-stiff 1} are critical points  of the Dirichlet energy
\begin{equation}
E(u)= \frac{1}{2} \int_A |\nabla u|^2 dx
\end{equation}
in the space 

\begin{equation}
\I=\{u\in H^1(A,\mathbb{C}); |u| =1  \ \text{a.e. on} \ \partial A \}.
\end{equation}

$\I$ is not a vector space, and it does not have an obvious structure of smooth Banach manifold. By critical 
point we mean critical with respect to variations of the form $u_t=u+t\varphi$ for all $\varphi$ in 
$\mathcal{C}_c^\infty(A, \C)$ and $u_t=ue^{it\psi}$ for all $\psi$ in $\mathcal{C}^\infty (\overline{A},\R)$. It is classic that the 
first type of variations gives that $\Delta u=0$ in $A$ and we can see that the second type of variations implies 
that 

\begin{equation}\label{condition3}
u\wedge \partial_\nu u = 0, \ \text{a.e.} \  \text{on} \ \partial A.
\end{equation} 

Let us explain in more details what the condition \eqref{condition3} means. We can see that a solution $u$ of 
\eqref{semi-stiff 1} is sufficiently regular (c.f. proposition \ref{regularity}), then we can write locally near 
the boundary $u=|u|e^{i\varphi}$. The boundary conditions give us that Dirichlet boundary conditions are 
prescribed for the modulus, $|u|=1$ on $\partial A$ and homogeneous Neumann conditions are prescribed for the 
phase, $\frac{\partial \varphi}{\partial \nu}=0$ on $\partial A$. In the literature these boundary conditions are 
called ``semi-stiff" because it has been first studied in the context of Ginzburg-Landau (G.L in short) equations 
and in this context the Dirichlet problem is named ``stiff" problem whereas the homogeneous Neumann problem is 
called ``soft".\\

Solutions of \eqref{semi-stiff 1} are linked to the notion of \textit{$\frac{1}{2}$-harmonic 
maps} defined by F.Da 
Lio and T.Rivière in \cite{DaLio}. 
\begin{definition}
Let $U$ be a smooth bounded open  set in $\R^2$, $g: \partial U \rightarrow \mathbb{S}^1$ is a \textbf{$\frac{1}
{2}$-harmonic map} if it is a critical point, for variations of the form $g_t=ge^{it\psi}$ with $\psi \in H^{1/2}
(\partial U,\R)$, of the functional
$$F(g)= \|g\|_{\dot H^{1/2}(\partial U)}^2 =\frac{1}{2}\int_U |\nabla \tilde{u}(g)|^2,$$

\noindent where $\|g\|_{\dot H^{1/2}(\partial U)}$ is the homogeneous $H^{1/2}$ Sobolev norm and $\tilde{u}(g)$ 
denotes the harmonic extension of $g$ in $U$.
\end{definition}

Equivalently, following \cite{Fraserr}, we can say that $g: \partial U \rightarrow \mathbb{S}^1$ is a $\frac{1}{2}$-harmonic map if there is a harmonic map 
$u:U \rightarrow \mathbb{D}$ such that $u=g$ on $\partial U$ and $u$ is stationary with respect to variations which 
preserve $\mathbb{D}$ but do not necessarily preserve $\mathbb{S}^1$. One can then see that $g: \partial A \rightarrow \C$ is a 
$\frac{1}{2}$-harmonic map if and only if its harmonic extension $\tilde{u}(g)$ is a solution of \eqref{semi-stiff 1}. We also note that the problem \eqref{semi-stiff 1} has some connections with the so-called Steklov 
eigenvalue problem. Indeed we can rewrite condition \eqref{condition3} as: there exists  $\sigma:\partial A 
\rightarrow \R$ such that 

\begin{equation}\label{condition4}
\partial_\nu u(x)= \sigma (x) u(x)\ \text{for all} \ x\in \partial A.
\end{equation}

The Steklov eigenvalue problem in an annulus $A$ consists in finding a function $f$ and a real number $\sigma$ 
such that 

\begin{equation}\label{Steklov}
\left\{
\begin{array}{rcll}
\Delta f& =& 0, &  \ \text{in} \ A, \\
\partial_\nu f & =& \sigma f, & \ \text{on} \ \partial A.
\end{array}
\right.
\end{equation}
If $(f,\sigma)$ exists then $\sigma$ is an eigenvalue of the Dirichlet-to-Neumann map. This is the map

$$L: \mathcal{C}^\infty(\partial A) \rightarrow \mathcal{C}^\infty (\partial A)$$
given by 
$$Lg=\partial_\nu \tilde{u}(g)$$

\noindent where $\tilde{u}(g)$ denotes the harmonic extension of $g$ in $A$. If it happens that $u$ is a solution of 
\eqref{semi-stiff 1} which satisfies the condition \eqref{condition4} with the function $\sigma$ which is 
constant then $(\text{Re}(u_{|\partial A}), \text{Im}(u_{|\partial A}))$ are both Steklov eigenfunctions  associated to the 
same Steklov eigenvalue and which satisfy 
$ \text{Re} (u_{|\partial A})^2 +\text{Im} (u_{|\partial A})^2=1$. We will see that if we add the hypothesis that $(\text{Re}(u_{|
\partial A}), \text{Im}(u_{|\partial A}))$ are Steklov eigenfunctions then we are able to describe solutions $u$ of 
\eqref{semi-stiff 1} (c.f. proposition \ref{linkwithSteklov}).  The Steklov eigenvalue problem and $\frac{1}{2}$-
harmonic maps are linked to free boundary minimal surfaces in the euclidean ball $B^n$ for further references on 
this topic see \cite{Fraser}, \cite{Fraserr}, \cite{Schoen}.\\

The physical motivation for studying semi-stiff boundary conditions comes from the following fact. While they 
were examining the  Ginzburg-Landau model, the authors of \cite{BBH} suggested to consider a simplified Ginzburg-
Landau functional
\begin{equation}
E_\e(u)= \frac{1}{2}\int_G|\nabla u|^2 +\frac{1}{4\e^2}(1-|u|^2)^2dx
\end{equation}
and to prescribe a Dirichlet boundary condition $u=g$ on the boundary $\partial G$ (in their work $G$ was a 
simply connected domain). They observed that a boundary data such that $|u|=1$ on $\partial G$ and $\deg(g,
\partial G)=d$ creates ``[...]the same ``quantized vortices" as a magnetic field in type-II superconductors or as 
angular rotation in superfluids." 
However they mentioned that ``in physical situations the Dirichlet condition is not realistic" because in the G.L 
theory only $|u|^2$ has a physical meaning (it is the density of Cooper pairs of electrons see for example 
\cite{ESSS}). That is why, some years later in the work \cite{berlyand2002symmetry} the authors tried to relax 
this condition by imposing only the condition $|u|=1$ on the boundary. 
 The Dirichlet energy can be viewed as a limit, when $\e$ goes to infinity of the Ginzburg-landau energy and it
was studied by L.Berlyand, V.Rybalko, P.Mironescu, E.Sandier in the work \cite{BRMS} in the case of a simply connected  domain.\\

Functions of the space $\I$ are classified by their \textit{degrees} on the two boundaries of the domain. If $g \in 
\mathcal{C}^1(\Gamma, \mathbb{S}^1)$, where $\Gamma$ is a smooth, simple, closed curve then by definition

\begin{equation}
\deg(g,\Gamma)= \int_{\Gamma} g \wedge \partial _\tau g d\tau.
\end{equation}

We can still define the degree for maps in $H^{\frac{1}{2}}(\Gamma,\mathbb{S}^1)$. We refer to section 2 for more information and references on the degree. We set 
\begin{equation}
\I_{p,q}=\{ u \in \I ; \deg(u,\partial \Omega)=p \ \text{and} \ \deg(u,\partial \omega)=q \}
\end{equation}
\noindent and 
\begin{equation}
m(p,q)= \inf\{E(v) ; v\in \I_{p,q}\}.
\end{equation}
 
The main results of this paper describe the existence and non-existence of solutions of \eqref{semi-stiff 1} in 
each $\I_{p,q}$, $(p,q)\in \mathbb{Z}²$ with a special emphasis on minimizing solutions (i.e. minimizers of the 
Dirichlet energy $E$ in spaces $\I_{p,q}$). Let us recall that the degree is not continuous under the weak 
convergence in $H^{\frac{1}{2}}$ as shown by the following example.

\begin{example}
Let $M_n:\mathbb{D} \rightarrow \mathbb{D}$ defined by  $M_n(z)= \frac{z-(1-1/n)}{1-(1-1/n)z}$, then $M_n \rightharpoonup -1$ 
weakly in $H^1$, 
$\deg(M_n(z),\mathbb{S}^1)=1$   for all $n \in \mathbb{N}$ but $\deg(-1,\mathbb{S}^1)=0$.
\end{example}

Hence we can not use the direct methods of calculus of variations. We are in presence of a problem of 
\textbf{lack of compactness}. The same phenomenon occurs for the study of the G.L equations with semi-stiff 
boundary conditions which were studied in \cite{berlyand2002symmetry},\cite{Capacity}, \cite{Nonexistence},\cite{uniqueness},\cite{unpublished},\cite{Solutions},\cite{ multiplyconnected}, \cite{BRMS},
\cite{existenceofminmaxcritical} and \cite{Size}.
\\

In the case of a simply connected domain $\Omega$, in \cite{BRMS} the authors obtained all the critical points of 
the Dirichlet energy with prescribed degrees. Using the conformal invariance of the Dirichlet energy and the 
Riemann's theorem they assumed that $\Omega=\mathbb{D}$ and $\partial \Omega=\mathbb{S}^1$. Recall that a 
Blaschke product is a map of the form 

$$ B_{\alpha,z_1,...,z_d}= \alpha \prod_{j=1}^d \frac{z-z_j}{1-\overline{z_j}z}, \ \ z\in \overline{\mathbb{D}}, 
\ \alpha \in \mathbb{S}^1, \ z_j \in \mathbb{D}, \ \text{for} \ j=1,... d.$$ Then using a lemma similar to 
\ref{Comparison1} and  a tool called the Hopf quadratic differential (c.f. section 3) they proved:

\begin{theorem}[\cite{BRMS}]\label{BRMS1}

The critical points of $E$ in $\I_d=\{u \in \I; \deg(u,\mathbb{S}^1)=d\}$ are precisely
\begin{itemize}
\item[a)] the $d$-Blaschke products if $d>0$,
\item[b)] the conjugates of $(-d)$-Blaschke products if $d<0$,
\item[c)] constant of modulus 1 if $d=0.$
\item[d)] All these solutions are minimizing in $\I_d$.

\end{itemize}  
\end{theorem}

In the case of a doubly connected domain $A$, the conformal invariance of the energy $E$ allows us to assume that 
$A=\{z\in \C ; \varrho < |z| < 1 \}$, where $\frac{2\pi}{\ln(1/\varrho)}=\capa(A)$ and $\capa(A)$ is the capacity of 
the domain (c.f. section 2 for a precise definition of the capacity). In the case $p=q=1$, L.Berlyand and 
P.Mironescu in \cite{unpublished} have showed that 

\begin{proposition}[ \cite{unpublished}]\label{BerlyandMironescu}
The only minimizers of $E$ in $\I_{1,1}$ are of the form
$$u_1(z)=\alpha\frac{1}{1+\varrho}(r+\frac{\varrho}{r})e^{i\theta}, $$
where $\alpha$ is a constant of modulus one. Moreover
$$m(1,1)= 2\pi\frac{1-\varrho}{1+\varrho}< 2\pi.$$
\end{proposition}

In order to prove this result they used the fact that
\begin{equation}
m(p,q)= \inf_{g: \partial A \rightarrow \mathbb{S}^1} E(\tilde{u}(g))
\end{equation}

\noindent where $g$ satisfies  $\deg(g,\partial \Omega)=p$, $\deg(g,\partial \omega)=q$ and $\tilde{u}(g)$ is the harmonic 
extension of $g$ in $A$. This is the $\frac{1}{2}$-harmonic maps point of view. Their approach consisted in 
writing $E(\tilde{u}(g))$ for any $g$ boundary data with degree one on the two boundaries in function of the 
Fourier coefficient of $g$. Then they minimized that quantity under the constraint $\deg(g,\partial 
\Omega)=\deg(g,\partial \omega)=1$ directly, using the expression of the degree in terms of the Fourier 
coefficients of the boundary data (see \cite{oldanew} for a formula of the degree involving Fourier coefficient). 
We can not use the same method when $p>1$ or $p\neq q$.  \\

We note that we do not need to study all classes $\I_{p,q}$. It is possible to reduce the number of cases to study. Indeed, thanks to the properties of the degree (c.f. lemma \ref{propertiesofdegree}), one can see that if $u$ is a minimizer of $E$ in $\I_{p,q}$ then $\bar{u}$ is a minimizer of $E$ in $\I_{-
p,-q}$, and $u(\frac{\sqrt{\varrho}}{\bar{z}})$ is a 
minimizer of $E$ in $\I_{q,p}$ (this latter fact is true because of the conformal invariance of the Dirichlet 
Energy). Thanks to this remark we can restrict ourselves to three different cases: $p\geq 0\geq q$, $p=q>0$,  $p>q>0$.

The main results of the paper are the followings:

\begin{theorem}\label{theorem1}
Let $p \geq 0 \geq q$ then $m(p,q)= \pi(p+|q|)$ and 
\begin{itemize}
\item[1)] if $p=q=0$ then the only minimizing solutions of \eqref{semi-stiff 1} are constants in $\mathbb{S}^1$.
\item[2)] If $p>0$ and $q=0$ then there is no solution of \eqref{semi-stiff 1} in $\I_{p,0}$, in particular there 
is no minimizer.
\item [3)] If $p>0>q$ then there exist an infinite number of solutions of \eqref{semi-stiff 1}, all solutions are holomorphic and 
energy minimizing, i.e. m(p,q) is attained.
\end{itemize}
\end{theorem}

\begin{theorem}\label{theorem2}
Let $p=q\geq 2$ then there exist critical values of $\varrho$ called $\varrho_p$ and $\varrho_p'$ such that $\varrho_{p}' 
\leq \varrho_p$ and
\begin{itemize}
\item[1)] If $\varrho > \varrho_p$ then $m(p,p)$ is attained by a unique (modulo rotations) radially symmetric 
minimizer. Hence $m(p,p)=2\pi p\frac{1-\varrho^p}{1+\varrho^p}$.
\item[2)] If $\varrho < \varrho_p'$ then the radially symmetric solution $u_p(z)=\frac{1}{1+\varrho^p}(r^p+\frac{\varrho^p}
{r^p})e^{ip\theta}$ is not minimizing. Furthermore it holds that $\sqrt{2}-1 =\varrho'_2<\varrho'_3<...<\varrho'_p 
<...<1$.
\item[3)] In the case $p=q=2$, we have that if $ \varrho'_2 \leq \varrho \leq \varrho_2$ then $m(2,2)$ is attained.
\end{itemize}
\end{theorem}

\textbf{Remark:} In point 3) of the previous theorem we do not know if $m(2,2)$ is attained by the radially symmetric solution $u_2$, neither if minimizers are unique (up to rotations).

The existence of minimizers of $E$ depends strongly on the prescribed degrees. We observe that it also depends on 
the capacity of the domain. The role of the capacity of the domain was already pointed out for the Ginzburg-
Landau energy  in the articles \cite{berlyand2002symmetry}, \cite{uniqueness},\cite{Capacity}, 
\cite{unpublished}, \cite{Size}. However in the degree one case for the Dirichlet energy the ``size" of the domain does not play a role in the existence of minimizers.

\begin{theorem}\label{theorem3}
Let $p>q>0$ then  there is no solution of \eqref{semi-stiff 1} in $\I_{p,q}$. In particular there is no minimizer 
of $E$ in $\I_{p,q}$ and we have $m(p,q)=m(q,q)+\pi(p-q)$. 
\end{theorem}

Since radial solutions are not always minimizing for $p=q\geq2$ one can wonder if other non-radial solutions of \eqref{semi-stiff 1} exist in $\I_{p,p}$. We obtained non radially symmetric solutions of \eqref{semi-stiff 1} in $\I_{p,p}$  which could be minimizer of $E$ in $\I_{p,p}$. 

\begin{theorem}\label{nonrad}
There exist non radial solutions of \eqref{semi-stiff 1} in $\I_{p,p}$ if $p \geq 2$.
\end{theorem}

However we do not know if these solutions are indeed minimizing or even if a minimizer of $E$ in $\I_{p,p}$ exists when the radial solution $u_p$ is not minimizing. 

The essential tool we used to obtain these results are the so-called  Hopf quadratic differential (cf. definition 
\ref{Hopff}). Using the three conditions of equation \eqref{semi-stiff 1} we can prove that the Hopf differential 
has the following form 

\begin{equation}
Q(u)=\H_u(z)(dz)^2=\frac{c}{z^2}(dz)^2, \ \text{in} \ A
\end{equation}

\noindent with $c$ a constant real number. We also use a deep relation between harmonic maps and minimal 
surfaces. This link is made thanks to the Hopf differential. Roughly speaking given a harmonic map $u$ such that 
$\Delta u =0$, locally we can find a harmonic function $h$ such that $X=(u,h)$ is a conformal parametrization 
of a minimal surface and the function $h$ is given in terms of the Hopf differential (see lemma \ref{lifting}) 
for a precise statement). It turns out that finding solutions of \eqref{semi-stiff 1} with $\H_u(z)=\frac{c}
{z^2}(dz)^2$ where $c<0$ is equivalent to finding minimal surfaces bounded by two concentric $p$-coverings of circles in parallel 
planes. We call $p$-covering of  a circle a parametrization of a circle of degree $p$. In 1956 in a beautiful paper \cite{Shiffman} M.Shiffman proved that if $S$ is a minimal surface bounded by two concentric 
circles in parallel planes then $S$ is (part of) a catenoid.
However in this theorem we assume that the circles are described only once. In terms of solutions of 
\eqref{semi-stiff 1} it is equivalent to ask that $\deg(u,\mathbb{S}^1)= \deg(u,C_\varrho)=1$. In order to find non-radially symmetric solutions of \eqref{semi-stiff 1}, with $\deg(u,\mathbb{S}^1)= \deg(u,C_\varrho)=p \geq 2$  one 
can look 
for immersed minimal surfaces bounded by two $p$-coverings of circles in parallel planes that are not rotationally symmetric. We 
obtained such surfaces by bifurcation of a $p$-covering of catenoids. Then thanks to the link between equation 
\eqref{semi-stiff 1} and the minimal surface problem we deduce theorem \ref{nonrad} from

\begin{theorem}\label{theorem4}
There exist non rotationally symmetric, immersed minimal surfaces in $\R^3$ bounded by two concentric $p$-coverings of circles in 
parallel planes if $p\geq 2$. 
\end{theorem}

There are some results in the literature concerning bifurcation of constant mean curvature (CMC) submanifolds. In 
particular in \cite{KPP} the authors studied bifurcation of (compact portions of) CMC nodoids in $\R^3$ whose 
boundary consists of two fixed coaxial circles of the same radius lying in parallel planes. This situation 
presents some similarity with our problem. However we used different techniques to obtain bifurcation of 
catenoids. The novelty in our approach is to consider bifurcation of $p$-coverings of catenoids. If $p \geq 2$ we 
can apply the theorem of Crandall-Rabinowitz (see \cite{Crandall}) to prove that bifurcation occurs and produces 
non rotationally symmetric minimal surfaces.\\

Let us mention that in a series of papers (see \cite{doublyconnectedminsurface}, 
\cite{nharmonicmappings}, \cite{mappingsofleastenergy} and the references therein) T. Iwaniec and J.Onninen studied harmonic mappings with the 
same form of Hopf differential i.e. $\H_u(z)=\frac{c}{z^2}$ with $c$ real. One of their purpose was to minimize 
the Dirichlet Energy on an annulus among some class of homeomorphisms. The common feature of their problem and 
ours is that we are both interested in critical points of the Dirichlet energy among maps with a given homotopy 
class at the boundary. The main difference is that they considered one-to-one mapping while we allow maps which are not one to one. They also made a link between such mappings and minimal surfaces (see 
\cite{doublyconnectedminsurface}).\\

The paper is organized as follows: part 2 is devoted to known analytic results and to the 
analysis of lack of 
compactness for minimizing sequences of the energy, in section 3 we present the properties of 
the Hopf quadratic differential of solutions of our problem and make a link with minimal surfaces theory. 
In 
section 4 we study 
holomorphic solutions of the problem. Section 5 is devoted to a discussion on radial solutions 
of \eqref{semi-stiff 1}. In section 6 we prove theorem \ref{theorem4}.

\section{Preliminaries}
\subsection{Notations and definitions}
Throughout the paper we use the following notations:

\begin{itemize}
\item[*] The vectors $a=(a_1,a_2)$ are identified with complex numbers $a=a_1 +ia_2$.
\item[*] $a \wedge b$ stands for the vector product $a \wedge b =a_1b_2-a_2b_1=\frac{i}{2}(a
\bar{b}-\bar{a}b)=\det(a,b)$.
\item[*] $ \langle a, b \rangle$ stands for the scalar product $\langle a,b \rangle 
=a_1b_1+a_2b_2= \frac{1}{2} (a\bar{b}+\bar{a}b)$.
\item[*]$\mathbb{D}$ denotes the unit disc, $\mathbb{S}^1$ denotes the unit circle. More 
generally we set 
$\mathbb{D}_r:=\{z\in \C ; |z|< r\}$ and $C_r:=\{z \in \mathbb{C}; |z|=r\}$.
\item[*] $A_r:=\{z \in \C ; \ r<|z|<1 \}$.
\item[*] The orientation of simple curves (in particular $\partial \omega$ and $\partial \Omega$) in 
$\R^2$ is assumed to be counter-clockwise. We denote by $\tau$ the unit tangent vector pointing 
counter-clockwise and $\nu$ the outer unit normal vector, hence $(\nu,\tau)$ is direct on $
\partial \Omega$ and $(\nu,\tau)$ is indirect on $\partial \omega$.
\end{itemize}

The main ingredient of the functional setting adapted to the study of critical points of $E$ in 
$\I$ is the degree. Let us recall this notion, we refer to \cite{oldanew}, 
\cite{Newquestiondegree}
and the references therein for more on this subject.
Let $\Gamma$ be a smooth, simple, closed curve. 
\begin{definition}
Let $g:\Gamma \rightarrow \mathbb{S}^1$, the \textbf{degree} of $g$ is the following quantity:
$$\deg(g,\Gamma)=\frac{1}{2\pi}\int_\Gamma g \wedge \partial_\tau g d\tau . $$
\end{definition}

The degree of a function is an integer and it measures the algebraic change of phase 
of $g$. A.Boutet de Monvel and O.Gabber have noticed that we can still define a degree for maps 
$g \in H^{\frac{1}{2}} (\Gamma,\mathbb{S}^1)$ (see \cite{Ogabber} and \cite{oldanew}). This 
degree is defined by approximation, indeed $\mathcal{C}^\infty(\Gamma,\mathbb{S}^1)$ is dense in  $H^
\frac{1}{2} (\Gamma,\mathbb{S}^1)$ and we can see that the degree is continuous with respect to 
the strong $H^{\frac{1}{2}}$ convergence. Here are well-known properties of the degree (c.f. \cite{oldanew} or \cite{Newquestiondegree}):

\begin{lemma}\label{propertiesofdegree}
Let $g,h\in H^{1/2}(\Gamma, \mathbb{S}^1)$. Then the following hold.
\begin{itemize}
\item[1)] If $g$ is continuous, then the degree of $g$ in the sense $H^{1/2}$ maps is the same as the degree of $g$ in the sense of continuous maps.
\item[2)] $\deg(gh)=\deg(g)+ \deg(h)$.
\item[3)] $\deg(\overline{g})=-\deg(g)$.
\item[4)] $\deg(g/h)= \deg(g)-\deg(h)$.
\end{itemize}
\end{lemma}

The degree can be used to characterize the connected 
components of $\I$. These are exactly the sets $\I_{p,q}=\{ u \in \I ; \deg(u,\partial 
\Omega)=p \ \text{and} \ \deg(u,\partial \omega)=q \},$ defined in introduction:

\begin{proposition}[\cite{Ogabber}]
The $\I_{p,q}$ are the connected components of $\I$. They are open and closed in $\I$ for the topology induced by the $H^1(A)$ norm.
 $$\I =\bigcup_{p,q \in \mathbb{Z}^2} \I_{p,q}.$$
\end{proposition}

This proposition allows us to say that if a minimizer for the energy $E$ in the class $\I_{p,q}$ exists then it is a local minimizer of the energy in $\I$ and hence a solution of \eqref{semi-stiff 1}. As stated in the introduction the degree is not continuous under the weak $H^1$ convergence. That is why finding solutions of \eqref{semi-stiff 1} is a non trivial problem.\\

The statement of theorem \ref{theorem2} shows that the capacity of an annulus is an essential notion in the study of problem \eqref{semi-stiff 1}. \begin{definition}
Let $A$ be a doubly connected domain. Let $V \in \mathcal{C}^{\infty}(A, \mathbb{R})$ be a solution of 

\begin{equation}
\left\{
\begin{array}{rcll}
\Delta V&=&0, \ \ \text{in} \ A, \\
V&=& 1, \ \ \text{on} \ \partial \Omega,\\
V&=& 0, \ \ \text{on} \ \partial \omega. 
\end{array}
\right.
\end{equation}
then
\begin{equation}
\capa(A):=\frac{1}{2}\int_A |\nabla V|²dx.
\end{equation}
\end{definition}

The capacity of a domain is preserved under conformal transformations and for an annulus $ G=\mathbb{D}_R \setminus \mathbb{D}_r$ it holds $\capa(G)= \frac{2\pi}{\ln(R/r)}$. The capacity measure the ``thickness" of the domain. Every annular domain $A$ is conformally equivalent to an annulus $A_\varrho=\{ z \in \C ;\varrho <|z|<1\}$ with $\frac{2\pi}{\ln(1/\varrho)}=\capa(A)$. 

Many lemmas collected in the rest of this section were first proved for the G.L energy, or the G.L equation their proof can be found in \cite{Capacity},\cite{unpublished},\cite{multiplyconnected}. Their adaptation to the Dirichlet energy is straightforward.\
\subsection{Properties of solutions of \eqref{semi-stiff 1} }
In this subsection we state two classical results concerning elliptic partial differential equations: regularity and maximum principle. 
\begin{proposition}\label{regularity}
Let $u$ be a solution of \eqref{semi-stiff 1} then $u$ is in $\mathcal{C}^{\infty}(\overline{A})$.
\end{proposition}

\begin{proof}
Since we have $\Delta u =0$ in $A$, $u$ is harmonic and a classical result says that $u\in \mathcal{C}^{\infty}(A)$.
It is the smoothness up to the boundary which is nontrivial and for that we refer to lemma 4.4 in \cite{unpublished}.
\end{proof}

\begin{proposition}
Let $u$ be a solution of \eqref{semi-stiff 1} then 
\begin{equation}
|u| \leq 1, \ \ in \ A.
\end{equation}
\end{proposition}

The latter proposition is an application of the maximum principle and can be found in \cite{BBH} for the G.L equation.

 \subsection{Minimizing sequences: price lemma and insertion of bubbles}

The lemmas presented here are very similar to those in \cite{kuwert}. It is because the problem studied in \cite{kuwert} presents some similarity with our problem.
First we can construct test functions that give us some estimation about the value of $m(p,q)$.

\begin{lemma}(\cite{multiplyconnected})\label{Comparison1}
For $u\in \I_{r,s}$ , $(p,q)\in \mathbb{Z}\times \mathbb{Z}$ and $\delta>0$ there exists $v\in \I_{p,q}$  such that 
$$ E(v) \leq E(u)+\pi(|p-r|+|q-s|)+\delta.$$ In particular

$$m(p,q) \leq m(r,s)+ \pi(|p-r|+|q-s|).$$
\end{lemma}

We can deduce from this lemma, since $m(0,0)=0$ that
\begin{equation}
m(p,q) \leq \pi(|p|+|q|). 
\end{equation}

The next lemma gives us some information about the ``cost" for a weak limit to jump in another class.

\begin{lemma}[\cite{unpublished},\cite{Capacity}]\label{Price lemma1}
Let $\{u^{(n)}\} \subset \I_{p,q}$ be a sequence that converges to $u$ weakly in $H^1(A,\R^2)$ with $u\in \I_{r,s}$. Then
$$ E(u) \leq \liminf_{n \rightarrow + \infty} E(u^{(n)})-\pi(|p-r|+|q-s|).$$
\end{lemma}

These two lemmas used together allow us to make a first description of what can happen to a minimizing sequence in $\I_{p,q}$ for $E$. The proof of the following lemma is inspired by \cite{kuwert}.

\begin{lemma}\label{behavior1}
Assume that a minimizing sequence  $\{u^{(n)}\} \subset \I_{p,q}$ for $m(p,q)$ converges weakly to some $u\in \I_{r,s}$. Then
$$ E(u)=\liminf_{n\rightarrow +\infty} E (u^{(n)}) -\pi(|p-r|+|q-s|),$$
and $u$ minimizes energy in $\I_{r,s}$ that is 
$$E(u)= m(r,s).$$
\end{lemma}

\begin{proof}
Thanks to lemma \ref{Price lemma1} we have
$$ E(u) \leq \liminf_{n \rightarrow + \infty} E(u^{(n)})-\pi(|p-r|+|q-s|).$$ But thanks to lemma \ref{Comparison1}
$$E(u) \geq m(r,s) \geq m(p,q)-\pi(|p-r|+|q-s|)$$
thus $$E(u)=m(p,q)-\pi(|p-r|+|q-s|)$$ and we can apply lemma \ref{Comparison1} again to obtain
$$m(r,s) \leq E(u) \leq m(r,s)$$ hence
$$E(u)= m(r,s).$$
\end{proof}
With the notation of lemma \ref{behavior1}, it means that if the infimum $m(p,q)$ is not attained then the weak limit of a minimizing sequence falls into an another class where the infimum is attained. 

To conclude this section let us state a stronger version of lemma \ref{Comparison1} which allows us to give a little better description of the behavior of the minimizing sequences.

\begin{lemma}(\cite{Insertion} )\label{strong bubble}
Let  $u\in \I_{p,q}$ be a solution of the Laplace Equation with semi-stiff boundary conditions and $k\in \mathbb{N}^*$. 
\begin{itemize}
\item[i)] Assume that there is $x_0 \in \partial \Omega$ such that
$$u\wedge \partial_{\tau}u(x_0) > -\langle u , \partial_{\nu}u\rangle (x_0), $$ then there exists $v\in \I_{p-k,q}$ such that 
$$E(v)< E(u)+k\pi.$$
\item[ii)] Assume that there is $x_0 \in \partial \Omega$ such that 
$$u\wedge \partial_{\tau}u(x_0)< \langle u , \partial_{\nu}u \rangle (x_0) ,$$ then there exists $v\in \I_{p+k,q}^d$ such that
$$E(v)< E(u)+k\pi.$$
\end{itemize}
\end{lemma}

The proof of this result can be found in \cite{Insertion}.

\begin{remark} If one writes locally near the boundary $u=\varrho e^{i\varphi}$ then 

$$ u\wedge \partial_{\tau}u(x_0) = \partial_\tau \varphi (x_0) $$
$$ \langle  u , \partial_{\nu}u \rangle (x_0)= \partial_\nu \varrho (x_0). $$
\end{remark}
\begin{lemma}\label{fall}
Let $(u_n) \subset \I_{p,q}$ be a minimizing sequence for $E$,  ($\lim_{n \rightarrow \infty}E(u_n)=m(p,q)$), 
up to extraction we have $u_n$ converges weakly to some $u$ in $H^1$ and $u \in \I_{r,s}$ for some $(r,s) \in \mathbb{N}^2$.

\begin{itemize}
\item[1)] If $p>0$ then $r\leq p$.
\item[2)] If $q>0$ then $s\leq q$.
\item[3)] If $p<0$  then $r\geq p$.
\item[4)] If $q<0$ then $s \geq q$.
\end{itemize}
\end{lemma}

\begin{proof}
Let us suppose that $p>0$. By contradiction if $r>p$ then (if we write locally $u=\varrho e^{i\varphi}$)
\begin{equation}\label{condition}
\frac{\partial \varphi}{\partial \tau}(x) \leq - \frac{\partial \varrho}{\partial \nu}(x) \ \ \ \forall x \in \mathbb{S}^1.
\end{equation}
Indeed assume that \eqref{condition} is not true then thanks to the lemma \ref{strong bubble}
you can find $v \in \I_{p,q}$ such that 
$$E(v) < E(u)+\pi(|p-r|+|q-s|).$$
However thanks to the lemma \ref{behavior1} we have  

$$E(u)= m(p,q) -\pi(|p-r| +|q-s|)$$ and then
$$E(v)< m(p,q)$$
this is a contradiction since $v \in \I_{p,q}$. \\

\noindent Furthermore the Hopf maximum principle tells us that $\frac{\partial \varrho}{\partial \nu }\geq 0$ on $\partial A$ thus
$$\frac{\partial \varphi}{\partial \tau} \leq 0 \ \text{on} \ \partial \Omega$$ and by integrating over $\partial \Omega$ we find that 

$$ 2\pi r \leq 0.$$
The proof is the same for the other cases.
\end{proof}

\section{Hopf differentials of solutions of \eqref{semi-stiff 1}}

In this section we present the main tool used to prove the results of this paper: the Hopf 
quadratic differential. We refer to \cite{Hopfdifferential} for properties of the Hopf differential.

\begin{definition}\label{Hopff}

Let $u: A \rightarrow \C$, the \textbf{ Hopf quadratic differential} of $u$ is 
\begin{equation}\label{Hopf}
\begin{array}{rclll}
Q(u)&=&\H_u(z)(dz)^2&=&\frac{1}{4}\big[|\partial_xu|^2-|\partial_yu|^2-2i \langle \partial_xu,\partial_yu 
\rangle \big](dz)^2\\
& & &=& (\partial_zu) (\overline{\partial_{\bar{z}} u)})(dz)^2 \\
& & &=& (\partial_zu)(\partial_{z} \bar{u})(dz)^2
\end{array}
\end{equation}
\end{definition}

\begin{proposition}\label{Hopf differential}
Let $u:A \rightarrow \C$.
\begin{itemize}
\item[1)] If $u$ is harmonic ($\Delta u=0$), then $ \H _u$ is holomorphic.
\item[2)] $\H_u=0$ is equivalent to $u$ conformal (i.e. $u$ holomorphic or $u$ anti-holomorphic).
\end{itemize}
\end{proposition}

\begin{proof}

1) Assume that $\Delta u=0$. Recall that $\Delta v =4 \partial_{\bar{z}} \partial{_z} v$ and let us compute 
\begin{equation}\nonumber
\begin{array}{rcll}
\partial_{\bar{z}}\H_u(z)&=&4 (\partial_{\bar{z}}\partial_z)u \partial_z\bar{u}+4\partial_z u(\partial_{\bar{z}}\partial_z)\bar{u} \\
&=&\Delta u \partial_z\bar{u}+ \partial_z u \Delta \bar{u} \\
&=&0.
\end{array}
\end{equation}
Hence $\partial_{\bar{z}}\H_u(z)=0$ that is $\H_u$ is holomorphic.\\

2) $\H_u=0$ is equivalent to $ \langle \partial_xu,\partial_yu\rangle=0$ and $|\partial_xu|=|\partial_yu|$. This means precisely that the differential of $u$ (which is a $2\times 2$ matrix) is a similitude, and that is the definition of a conformal application.

\end{proof}

The conformal invariance of the Dirichlet energy allows us to work in an annulus which is a rotational symmetric domain, then we will work in polar coordinates $z=re^{i\theta}$, with $\varrho<r<1$ and $0\leq \theta <2\pi$. A simple computation shows that we can write 

\begin{equation}
4z^2\H_u(z)=r^2|\partial_ru|^2-|\partial _\theta u|^2-2ir \langle \partial_ru,\partial_\theta u \rangle.
\end{equation}

\begin{lemma}\label{Hopfcst}
Let $u$ be a solution of \eqref{semi-stiff 1}, then $z^2\H_u(z)=c \in \R$ in $A$.
\end{lemma}

\begin{proof}
We have that $u$ is smooth up to the boundary and that $z^2\H_u(z)$ is holomoprhic in $A$ thanks to proposition \ref{Hopf differential}. On $\partial \Omega=\mathbb{S}^1$ we have $\partial_\nu u=\partial_r u$ and $\partial_\tau u =\partial_\theta u$. Hence 
$$\langle \partial_r u,\partial _\theta u \rangle=0, \ \ \text{on} \ \mathbb{S}^1.$$
Indeed on $\mathbb{S}^1$ we have $u\wedge \partial_r u=0$ according to \eqref{semi-stiff 1} and $\langle u ,\partial_\theta u \rangle=0$ because $|u|^2=1$. The same method with $\partial_\nu u=-\partial_r u$ and $\partial_\tau u=\frac{1}{\varrho}\partial_\theta u$ on $C_\varrho$ leads to 
$$\langle \partial_r u,\partial _\theta u \rangle=0, \ \ \text{on} \ C_\varrho.$$

Thus $z\mapsto z^2\H_u(z)$ is holomorphic in $A$ and takes real values on $\partial A$. We can conclude that this function is real valued in all $A$. Indeed the imaginary part of $z^2\H_u(z)$ is harmonic in $A$ and null on $\partial A$ thus it is identically null on $A$. $z^2\H_u(z)$ being holomorphic and real-valued we deduce that it is constant in $A$. 
\end{proof}

Thanks to this lemma we can say that if $u$ is a solution of \eqref{semi-stiff 1} then 
\begin{equation}\label{constHopf}
r^2|\partial_ru|^2-|\partial_\theta u|^2=c \ \ \text{in} \ A
\end{equation}

\begin{equation}\label{orthoradial}
\langle \partial_ru,\partial_\theta u \rangle =0, \ \ \text{in} \ A.
\end{equation}

There is a strong link between  $\C$-valued harmonic function, their Hopf differential and minimal surfaces. In order to explain this link let us recall few facts about and minimal surfaces theory. We also refer to \cite{Hauswirth} for more explanations on the link between the Hopf quadratic differential, harmonic maps and minimal surfaces.

\begin{definition}
Let $V \subset \R² \simeq \C$ be a domain. $X: V \rightarrow \R³$ is \textbf{a conformal (or isothermal) parametrization} of a surface if $X$ is an immersion (i.e. $|\partial_xX \wedge \partial_yX|$ is never zero in $V$) and 
\begin{equation}\label{conformalityrelation}
\left\{
\begin{array}{rcll}
\langle \partial_x X,\partial_yX \rangle =0 \\
|\partial_xX|²= |\partial_y|² = 0
\end{array}
\right.
\end{equation}
\end{definition}

It is well-known( see \cite{minimalsurfaces} p.77) that we can represent every regular surface of class $\mathcal{C}^2$ by 
conformal parameters.

\noindent Now we take advantage of the complex variable. If we set $X=(u,h)=(u_1,u_2,h)$ then $u:V \rightarrow \C$ and the conformality relations \eqref{conformalityrelation} reduce to one complex equation

\begin{equation}
\partial_zu_1²+\partial_z u_2² +\partial_z h² =0.
\end{equation}

\noindent A direct computation shows that 
\begin{equation}
\begin{array}{rcll}
\partial_zu_1^2+\partial_zu_2^2 &= &(\partial_zu_1+i\partial_zu_2)(\partial_zu_1-i\partial_zu_2) \\
&=& \partial_zu \overline{\partial_{\bar{z}}u} \\
&=& \H_u.
\end{array}
\end{equation}

\noindent Hence the conformality relations mean that 

\begin{equation}\label{conformality2}
\H_u +\partial_z h^2=0
\end{equation}
thus we see how the Hopf quadratic differential appears in surface theory.

\begin{proposition}(\cite{doublyconnectedminsurface} )\label{confandhopf}
Let $V\subset \C$ be a domain and  $X=(u,h):V \rightarrow \C \times \R=\R^3$  be the conformal representation of a surface (not necessarily minimal). Then 
\begin{itemize}
\item[*] the function $\partial_zu\overline{\partial_{\overline{z}}u}$ admits a continuous branch of square root in $V$,
\item[*]for each smooth closed curve $\Gamma \subset V$ we have 

$$ \text{Re} \int_\Gamma -i\sqrt{\partial_zu \overline{\partial_{\overline{z}}u}}dz =0 $$
\item[*] the real isothermal coordinate is given by 
$$h=\text{Re} \int_z^{z_0} -2i\sqrt{\partial_zu \overline{\partial_{\overline{z}}u}} dz $$
where the line integral runs along any smooth curve $\gamma \subset V$ beginning at a given point $z_0\in V$ and terminating at $z$.
\end{itemize}
\end{proposition}

Since we are interested in harmonic $\C$-valued functions we now explain how they are linked to minimal surface theory.

\begin{proposition}(\cite{minimalsurfaces} p.72)\label{minsurfaceanddharmfunction}
Let $X=(u_1,u_2,h): V \rightarrow \R³$ be a conformal parametrization of a surface. Then this surface is minimal if and only if $\Delta X=0$ i.e. $\Delta u_1=\Delta u_2 =\Delta h =0$.
\end{proposition}

Now we can state a proposition which allows us to build a minimal surface from a $\C$-valued harmonic function and its Hopf differential .

\begin{proposition}\label{lifting}
Let $u:V \rightarrow \C$ be a complex harmonic function and $\H_u(z)= \partial_zu \overline{\partial_{\overline{z}}u}$ its Hopf quadratic differential. Then locally outside the zeros of odd order of $\H_u$ if we set 
\begin{equation}\label{Hopfdef}
h(z)= \text{Re} \int_{z_0}^z -2i \sqrt{\H_u(z)} dz
\end{equation}
where the  line integral runs along any smooth curve $\gamma \subset V$ beginning at a given point $z_0\in V$ and terminating at $z$, then 
$X=(u,h)$, defined locally, is the isothermal parametrization of a minimal surface.  
\end{proposition}

\begin{proof}
Locally, near a point $z_0$ such that $\H_u(z_0) \neq 0$ or  near every zero of even order of $\H_u$, in simply connected sub-domains, we can always define $h$ by the relation \eqref{Hopfdef}. \\
This function $h$ is harmonic because it is the imaginary part of an holomorphic function. 
Hence if we want to show that $X$ defines a minimal surface we only must show that $X=(u,h)$ is 
a conformal parametrization thanks to proposition \ref{minsurfaceanddharmfunction}.  We must 
then show that 
$$\H_u +\partial_z h²=0.$$
But this fact comes from the definition of $h$ and from the fact that if $U$ is a holomorphic 
function  then 
$$U'(z)= \partial_zU =2\frac{\partial \text{Re} (U)}{\partial z} $$
thus with the definition \eqref{Hopfdef} we find that 
$$ \partial_zh=i\sqrt{\H_u} $$
and 
then 
$$\H_u +\partial_z h²=0.$$
\end{proof}

\begin{remark} The surface is planar ($h \equiv 0$) if and only if $u$ is conformal. The global 
lifting exists provided the imaginary part of the integral in \eqref{Hopfdef} is single valued. 
Since we are interested in $\H_u(z)=\frac{c}{z^2}$ it will be the case if $c<0$, because in 
this case the imaginary part of this term will be $\ln |z|$, but not in the case where $c>0$ 
(the imaginary part of the term will be $\arg (z)$). 
\end{remark}

\section{The case $c=0$: holomorphic solutions}							

Let $u$ be a solution of \eqref{semi-stiff 1} then $\H_u(z)=\frac{c}{z^2}$. What can we say about solutions with $c=0$ i.e. what can we say about conformal solutions of \eqref{semi-stiff 1}? More precisely do such solutions exist? If yes in which $\I_{p,q}$? Are they minimizing? We will restrict ourselves to holomorphic solutions, the antiholomorphic case being obtained by conjugation.

We first recall a formula which will be very useful in this section (see also \cite{oldanew}).

\begin{proposition}\label{diffdegree}
Let $u \in \I_{p,q}$ then 
$$\int_A \partial_xu \wedge \partial_yu =\pi (p-q).$$
\end{proposition}

\begin{proof}
This is an application of the divergence formula. We have 

\begin{eqnarray}
\int_A \partial_xu \wedge \partial_y u&=&\frac{1}{2}\int_A [\partial_x(u\wedge \partial_y u)+\partial_y(\partial_x u\wedge u)] \nonumber \\
&=& \frac{1}{2}\int_{\partial A} [(u\wedge \partial_yu)\nu_1+(\partial_x u\wedge  u)\nu_2] \nonumber \\
&=& \frac{1}{2} \int_{\partial A} u \wedge (\nu_1\partial_y u-\nu_2 \partial_xu) \nonumber \\
&=& \frac{1}{2} \int_{\partial \Omega} u\wedge (\tau_1 \partial_x u  +\tau_2 \partial_y u )-\frac{1}{2} \int_{\partial \omega} u\wedge (\tau_1\partial_xu  +\tau_2 \partial_yu ) \nonumber \\
&=& \frac{1}{2} \int_{\partial \Omega} u \wedge \partial_\tau u -\frac{1}{2}\int_{\partial \omega} u \wedge \partial_\tau u \nonumber \\
&=& \pi(p-q). \nonumber
\end{eqnarray}
Where we used that $\tau_1=-\nu_2$ and $\tau_2=\nu_1$ on $\partial \Omega$ because $(\nu,\tau)$ is direct on $\partial \Omega$ whereas  $\tau_1=\nu_2$, $\tau_2=-\nu_1$ on $\partial \omega$ because $(\nu,\tau)$ is direct on this boundary.
\end{proof}
We note that if $u$ is holomorphic and $u$ belongs to $\I$ then the difference between the degrees at the boundaries gives exactly the number of zeros of $u$.

\begin{lemma}\label{numberofzeros2}
Let $u$ be holomorphic and $u \in \I_{p,q}$ then $p \geq q$ and  $u$ possesses exactly $p-q$ zeros in $A$ counted with their multiplicities.
\end{lemma}

\begin{proof}
 If $u$ is holomorphic in $A$ we have $\partial_x u \wedge \partial_y u \geq 0$ because the differential of $u$ is a direct similitude. Hence we obtain thanks to the previous lemma that $\int_A \partial_x u \wedge \partial_y u=\pi(p-q) \geq 0$. Now  since $u$ is holomorphic its zeros are isolated in $A$, and hence there is a finite number of zeros of $u$ in $A$ because $\overline{A}$ is compact. Let us denote by $N$ the number of zeros of $u$ in $A$ (counted with multiplicity) and by $m$ the number of distinct zeros. Let $z_1,...,z_m$ be the zeros of $u$ in $A$ and $r>0$ small enough for $D(z_i,r)$ to contain only $z_i$ as a zero of $u$ for all $1 \leq i \leq m$. We set $\mathcal{B}:= A \setminus \bigcup_{i=1}^m D(z_i,r)$. The function $u$ does not vanish in $\mathcal{B}$ so we can set $T:=\frac{u}{|u|}$ in $\mathcal{B}$. We then have

\begin{equation}
\int_{\mathcal{B}} \partial_xT \wedge \partial_yT = \deg(T,\partial \Omega)- \deg(T,\partial \omega)- \sum_{i=1}^m \deg(T, \partial D(z_i,r) )
\end{equation}
thanks to the divergence formula (this is a formula analog to \ref{diffdegree}). However because $T$ is $\mathbb{S}^1$ valued we have $\int_{\mathcal{B}} \partial_xT \wedge \partial_yT =0$ and hence we find that 

\begin{eqnarray}
\deg(T,\partial \Omega)-\deg(T,\partial \omega)&=&\sum_{i=1}^m \deg(T,\partial D(z_i,r)) \nonumber \\
\deg(u,\partial \Omega)-\deg(u,\partial \omega)&=&\sum_{i=1}^m \deg(T,\partial D(z_i,r)) \nonumber \\
p-q &=& \sum_{i=1}^m \deg(T,\partial D(z_i,r)). \label{numberofzeros} 
\end{eqnarray}

We now claim that $\deg(T,\partial D(z_i,r))$ is the multiplicity of $z_i$ as a zero of $u$. Indeed $\deg(T,\partial D(z_i,r))$ measures the algebraic change of phase of $T$ (and hence of $u$ because $T=\frac{u}{|u|}$) on $\partial D(z_i,r)$. The variation of the argument (or the phase) of $u$ is also given by 
$$\frac{1}{2\pi i} \int_{\partial D(z_i,r)} \frac{u'(z)}{u(z)}dz .$$
But the argument principle for holomorphic functions (cf. \cite{Stein} th.4.1 p.90) gives us that 
$\frac{1}{2\pi i} \int_{\partial D(z_i,r)} \frac{u'(z)}{u(z)}dz$ is the number of zeros of $u$ inside $D(z_i,r)$ counted with multiplicity thus it is equal to the multiplicity of $z_i$ as a zero of $u$. And thus \eqref{numberofzeros} implies that 
$$N=p-q.$$
\end{proof}
\begin{lemma}\label{obstruction}
 Let $u$ be holomorphic, $u\in\I_{p,q}$ then $p>0$ and $q<0$.
\end{lemma}

\begin{proof}
If $u$ is holomorphic in $A$ and $u\in \I_{p,q}$  thanks to the previous lemma $p \geq q$ and $u$ possesses  exactly $p-q$ zeros in $A$. Let $\varrho< \mu<1$ with $\mu$ close enough to 1 such that $u$ has no zero in $A_\mu:=\{z\in \mathbb{C}; \mu < |z|<1\}$. We set $G=\ln|u|$, the logarithm of the modulus of an holomorphic function which does not vanish is an harmonic function thus we have 
\begin{equation}
\left\{
\begin{array}{rclll}
\Delta G & =& 0, & \text{in} & A_\mu, \\
G&=&0, &  \text{on} &   \mathbb{S}^1, \\
G&=&\ln|u|, & \text{on} & C_\mu.
\end{array}
\right.
\end{equation}

Now the maximum principle for holomorphic functions tells us that $|u|<1$ inside the ring $A$, hence $G=\ln|u|<0$, on $C_\mu$. Applying the Hopf's lemma (see \cite{evans} p.330) we find that
$$\frac{\partial G}{\partial \nu}(x)>0, \ \ \forall x\in \mathbb{S}^1.$$ 

But $\frac{\partial G}{\partial \nu}(x)=\frac{\partial(\ln|u|)}{\partial \nu}=\frac{1}{|u|}\frac{\partial |u|}{\partial \nu}=\frac{\partial |u|}{\partial \nu}(x), \ \ \forall x\in \mathbb{S}^1$ (recall that $|u|=1$ on $\mathbb{S}¹)$. We also have 

$$\frac{\partial |u|}{\partial \nu}=\frac{1}{2|u|}\frac{\partial |u|^2}{\partial\nu}=\frac{1}{|u|}\langle u,\frac{\partial u}{\partial \nu} \rangle =\frac{1}{|u|} u\wedge\frac{\partial u}{\partial \tau}$$

\noindent the last equality being true because $(\nu,\tau)$ is direct on $\mathbb{S}^1$ then $\frac{\partial u}{\partial \nu}$ and $\frac{\partial u}{\partial \tau}$ are directly orthogonal since $u$ holomorphic. Hence
$$2\pi p=\int_{\mathbb{S}^1}u\wedge \partial_\tau u=\int_{\mathbb{S}^1}\frac{\partial G}{\partial \nu}>0$$ and $p>0$. We can apply a similar argument on the other boundary $C_\varrho$ but this time $(\nu,\tau)$ is indirect thus 
$\langle u,\frac{\partial u}{\partial \nu} \rangle =-u\wedge \frac{\partial u}{\partial \tau}$ and $q<0$.
\end{proof}

\begin{remark}
Let $u$ be a holomorphic function, if $u$ belongs to the space $\I$ then $u$ satisfies the system \eqref{semi-stiff 1}. Indeed 
\begin{itemize}
\item[*] $\Delta u =0$ because the real and imaginary parts of a holomorphic function are harmonic.
\item[*] Because of the fact that $|u|=1$ on $\partial A$ we have $ \langle u,\partial_\tau u \rangle=0$ on $\partial A$, but 
$\partial_\tau u$ and $\partial_\nu u$ are orthogonal on $\partial A$ because $u$ is holomorphic, hence 
$u\wedge \partial_\nu u=  \langle u,\partial_\tau u \rangle=0$ on $\partial \Omega$ and $u\wedge \partial_\nu u= -\langle u, \partial_\tau u \rangle =0$ on $\partial \omega$ (the sign depending on the orientation of $(\nu,\tau))$. 
\end{itemize}
This means that holomorphic functions are natural candidates for the problem \eqref{semi-stiff 1}.
\end{remark}

\begin{lemma}\label{minholo}
Let $p\geq 0 \geq q$ then $m(p,q)=\pi(p+|q|)$ and $u$ is a minimizer of $E$ in $\I_{p,q}$ if and only if $u$ is holomorphic.
\end{lemma}

\begin{proof}
If $p=q=0$ then constant solutions in $\mathbb{S}¹$ are minimizers. During the proof we always assume that $p>0\geq q$.
We have the following point-wise equalities
$$ \frac{1}{2} |\nabla u|^2= \partial_x u \wedge \partial _y u +2|\partial_{\overline{z}}u|^2$$
$$ \frac{1}{2} |\nabla u|^2= -\partial_x u \wedge \partial _y u +2|\partial_{z}u|^2.$$

Hence we have $E(u)=\frac{1}{2}\int_A |\nabla u|² \geq |\int_A \partial_x u \wedge \partial _y u |$
and an integration by parts (formula \ref{diffdegree}) gives $E(u) \geq \pi |p-q| $. Then if $p>0$ and $q\leq0$ 
$$ E(u) \geq \pi (p+|q|)$$
with equality if and only if $u$ is holomorphic. But  lemma \ref{Comparison1}  tells us that 
$$m(p,q)\leq \pi (p+ |q|)$$ 
thus we can conclude that $m(p,q)=\pi(p+|q|)$ and that if a minimizer exists it must be holomorphic. Conversely if $u$ is holomorphic thanks to the previous computation we find that $u$ minimizes the energy in $\I_{p,q}$.\\
\end{proof}
This lemma allows us to prove the following part of  point 2) of  theorem \ref{theorem1}. 

\begin{proposition}
There is no minimizer of $E$ in $\I_{p,0}$ with $p>0$. 
\end{proposition}

\begin{proof}
Indeed there is no minimizer in $\I_{p,0}$ with $p>0$ because if this minimizer exists it should be holomorphic.
But lemma \ref{obstruction} says that there is no holomorphic function in $\I_{p,0}$.
\end{proof} 
 A similar result was obtained with a different proof in \cite{BRMS} (see lemma 9.9 p.1001) but only in the case $\I_{1,0}$. The authors proved that there is no minimizer of $E$ in $\I_{1,0}$ because there is no holomorphic function in $\I_{1,0}$. \\

\begin{proposition}\label{holomorphicsolution}
If $p>0>q$ then there is an infinite number of critical points of $E$ in $\I_{p,q}$ and these are holomorphic solutions which can be written explicitly. Furthermore every solution is energy minimizing.
\end{proposition}

For the proof of this proposition we give an explicit formula for a solution and we will check that it satisfies all the desired properties. We explain in the remark below the heuristic of the derivation of this formula. Let $p>0>q$ be two integers. We choose $p-q$ points $x_1,x_2,..., x_{p-q}$ in $A$ which satisfy:
\begin{equation}\label{Conditionpointsholo}
 -\displaystyle{\sum_{i=1}^{p-q} \frac{\ln |x_i|}{\ln\varrho}=q}.
\end{equation}
Note that it is possible to realize these conditions since $0<\sum_{i=1}^{p-q} \frac{\ln |x_i|}{\ln\varrho}<p-q$ because $\varrho<|x_i|<1$ for all $1 \leq i\leq p-q $ and $p-q > |q|$.
We then set 
\begin{equation}\label{holomorphicsol}
u(z):= z^q  \prod_{i=1}^{p-q} |x_i|\frac{1-\frac{z}{x_i}}{1-z\overline{x_i}}\prod_{k=1}^{+\infty} \frac{(1-\frac{\varrho^{2k}z}{x_i})(1-\frac{\varrho^{2k}x_i}{z})}{(1-\frac{\varrho^{2k}}{z\overline{x_i}})(1-\varrho^{2k}z\overline{x_i})}.
\end{equation}

We want to show that 
\begin{itemize}
\item[1)] $u$ is holomorphic in $A$,
\item[2)] $u\in \I$ i.e. $|u|=1$ on $\mathbb{S}^1$ and on $C_\varrho$,
\item[3)] $\deg(u,C_\varrho)=q$ and $\deg(u,\mathbb{S}^1)=p$.
\end{itemize}

\noindent To this end we set 
\begin{equation}\label{f_x}
f_{x_i}(z):=\frac{1-\frac{z}{x_i}}{1-z\overline{x_i}}\prod_{k=1}^{+\infty} \frac{(1-\frac{\varrho^{2k}z}{x_i})(1-\frac{\varrho^{2k}x_i}{z})}{(1-\frac{\varrho^{2k}}{z\overline{x_i}})(1-\varrho^{2k}z\overline{x_i})}
\end{equation}
we can thus rewrite 
$$u(z)=z^q \prod_{i=1}^{p-q} x_i f_{x_i}(z).$$
We have 
\begin{lemma}\label{f_xprop}
For all $1 \leq i \leq p-q$ 
\begin{itemize}
\item[i)] $f_{x_i}$ is a meromorphic function on $\C$.
\item[ii)] $f_{x_i}$ has simple zeros at the points $\frac{x_i}{\varrho^{2k}}$ for all $k \in \mathbb{N}$ and $\varrho^{2k} x_i$ for all $k \in \mathbb{N}^*$.
\item[iii)]  $f_{x_i}$ has simple poles at the points $\frac{\varrho^{2k}}{\overline{x_i}}$ for all $k \in \mathbb{N}$ and $\frac{1}{\varrho^{2k} \overline{x_i}}$ for all $k \in \mathbb{N}^*$.
\item[iv)] $\overline{f_{x_i}(z)}=f_{\overline{x_i}}(\overline{z})$ for all $z\in \C$.
\item[v)] For all $z\in \C$ it holds $f_{x_i}(z)f_{\overline{x_i}}(\frac{1}{z})=\frac{1}{|x_i|^2}$ 
and $f_{x_i}(z)f_{\overline{x_i}}(\frac{\varrho^2}{z})=1$.
\end{itemize}
\end{lemma}
\begin{proof}
i) To prove the first point it suffices to prove that each infinite product converges and 
is a holomorphic function in $\C$. For example for $\prod_{k=1}^{+\infty}(1-
\frac{\varrho^{2k}z}{x_i})$ it holds that $|\frac{\varrho^{2k}z}{x_i}| < \varrho^{2k-1}|z|$ because $|
x_i|>\varrho$. Besides since $\varrho<1$ the sum $\sum_{k\geq 1} \varrho^{2k-1}|z|$ is finite for 
all $z \in \C$. This implies (see for example proposition  3.2 p.141 in \cite{Stein}) that 
$\prod_{k=1}^{+\infty}(1-\frac{\varrho^{2k}z}{x_i})<+\infty$ and this product is holomorphic 
in $\C$. The proof goes the same way for the three other infinite products in $f_{x_i}$. 
\\
Points ii),iii) and iv) are immediate from the definition of $f_{x_i}$. \\
The point v) is obtained via the following computations:
\begin{eqnarray}
f_{x_i}(z)f_{\overline{x_i}}(\frac{1}{z}) &=& \frac{(1-\frac{z}{x_i})(1-\frac{1}{\overline{x_i}z})}{(1-z\overline{x_i})(1-\frac{x_i}{z})}\prod_{k=1}^{+\infty} \frac{(1-\frac{\varrho^{2k}z}{x_i})(1-\frac{\varrho^{2k}x_i}{z})(1-\frac{\varrho^{2k}}{\overline{x_i}z})(1-\varrho^{2k}\overline{x_i}z)}{(1-\frac{\varrho^{2k}}{\overline{x_i}z})(1-\varrho^{2k}\overline{x_i}z)(1-\frac{\varrho^{2k}z}{x_i})(1-\frac{\varrho^{2k}x_i}{z})} \nonumber \\
&=&\frac{(1-\frac{z}{x_i})(1-\frac{1}{\overline{x_i}z})}{(1-z\overline{x_i})(1-\frac{x_i}{z})} \nonumber \\
&=& \frac{1}{|x_i|^2}
\end{eqnarray}
and

\begin{eqnarray}
f_{x_i}(z)f_{\overline{x_i}}(\frac{\varrho^2}{z}) &=& \frac{(1-\frac{z}{x_i})(1-\frac{\varrho^2}{\overline{x_i}z})}{(1-z\overline{x_i})(1-\frac{\varrho^2x_i}{z})} \prod_{k=1}^{+\infty} \frac{(1-\frac{\varrho^{2k}z}{x_i})(1-\frac{\varrho^{2k}x_i}{z})(1-\frac{\varrho^{2k+2}}{\overline{x_i}z})(1-\varrho^{2k-2}\overline{x_i}z)}{(1-\frac{\varrho^{2k}}{\overline{x_i}z})(1-\varrho^{2k}\overline{x_i}z)(1-\frac{\varrho^{2k-2}z}{x_i})(1-\frac{\varrho^{2k+2}\overline{x_i}}{z})} \nonumber \\
&=& \frac{1-\frac{z}{x_i}}{1-z\overline{x_i}}\times \frac{1-z\overline{x_i}}{1-\frac{z}{x_i}}\prod_{k=1}^{+\infty}  \frac{(1-\frac{\varrho^{2k}z}{x_i})(1-\frac{\varrho^{2k}x_i}{z})(1-\frac{\varrho^{2k}}{\overline{x_i}z})(1-\varrho^{2k}\overline{x_i}z)}{(1-\frac{\varrho^{2k}}{\overline{x_i}z})(1-\varrho^{2k}\overline{x_i}z)(1-\frac{\varrho^{2k}z}{x_i})(1-\frac{\varrho^{2k}x_i}{z})} \nonumber \\
&=& 1.
\end{eqnarray}
\end{proof}
\begin{proof}[Proof of proposition \ref{holomorphicsolution}]

We now go back to the proof of proposition \ref{holomorphicsolution}.  To prove that $u$ is holomorphic in $A$ it suffices to apply the previous lemma \ref{f_xprop} and to observe that $u$ has no pole in $A$ (the points $\frac{\varrho^{2k}}{\overline{x_i}}$ for $k\in \mathbb{N}$ and $\frac{1}{\varrho^{2k}\overline{x_i}}$ for $k\in \mathbb{N}^*$ are not in $A$). \\
We now prove the point 2) that is $u \in \I$. For all $z\in \mathbb{S}^1$ we have $\overline{z}=\frac{1}{z}$. Thus for all $z\in \mathbb{S}^1$:
\begin{eqnarray}
|u(z)|^2 &=& u(z) \overline{u(z)} \nonumber \\
&=& \prod_{i=1}^{p-q}|x_i|^2f_{x_i}(z) \overline{f_{x_i}(z)}. \nonumber
\end{eqnarray}
we use  points iv) and v) of lemma \ref{f_xprop} to obtain that for all $z\in \mathbb{S}^1$
\begin{eqnarray}
|u(z)|^2 &=&  \prod_{i=1}^{p-q}|x_i|^2f_{x_i}(z)f_{\overline{x_i}}(\overline{z}) \nonumber \\
&=&  \prod_{i=1}^{p-q}|x_i|^2f_{x_i}(z)f_{\overline{x_i}}(\frac{1}{z}) \nonumber \\
&=& \prod_{i=1}^{p-q}|x_i|^2 \frac{1}{|x_i|^2} \nonumber \\
&=&1.
\end{eqnarray}
Hence $|u|=1$ on $\mathbb{S}^1$. Likewise for all $z\in \C_{\varrho}$
\begin{eqnarray}
|u(z)|^2&=& u(z)\overline{u(z)} \nonumber \\
&=& |z|^{2q}\prod_{i=1}^{p-q}|x_i|^2f_{x_i}(z) \overline{f_{x_i}(z)} \nonumber \\
&=& \varrho^{2q}\prod_{i=1}^{p-q}|x_i|^2f_{x_i}(z)f_{\overline{x_i}}(\overline{z}) \nonumber \\
&=&\varrho^{2q}\prod_{i=1}^{p-q}|x_i|^2f_{x_i}(z)f_{\overline{x_i}}(\frac{\varrho^2}{z}) \nonumber \\
&=& \varrho^{2q}\prod_{i=1}^{p-q}|x_i|^2 \nonumber \\
&=&1
\end{eqnarray}
the last equality being true because of the choice \eqref{Conditionpointsholo} that is $q=\displaystyle{-\sum_{i=1}
^{p-q} \frac{\ln |x_i|}{\ln \varrho}}$ and $\displaystyle{\prod_{i=1}^{p-q} x_i^2} \in \R^+$. We then have 
$u\in \I$. \\
To conclude the proof we only need to show that $\deg(u,C_\varrho)=q$. Indeed since $u$ has $p-q$ zeros in $A$ and since $\deg(u,\mathbb{S}^1)-\deg(u,C_\varrho)= \text{number of zeros of u in A}$ (c.f. proposition \ref{numberofzeros2}) we will deduce that $\deg(u,\mathbb{S}^1)=p$.  For all $N \in \mathbb{N}^*$ we set 
\begin{equation}\nonumber
u_N(z)=z^q\prod_{i=1}^{p-q}x_i  \frac{1-\frac{z}{x_i}}{1-z\overline{x_i}}\prod_{k=1}^{N} \frac{(1-\frac{\varrho^{2k}z}{x_i})(1-\frac{\varrho^{2k}x_i}{z})}{(1-\frac{\varrho^{2k}}{z\overline{x_i}})(1-\varrho^{2k}z\overline{x_i})}.
\end{equation}
The functions $u_N$ are meromorphic on $\C$. We count the number of zeros and poles of 
$u_N$ inside the disk $\mathbb{D}_\varrho$. The zeros of $u_N$ inside this disk are at the 
points $\varrho^{2k}x_i$ for $1\leq k \leq N$. We can see that $0$ is a pole of order $|q|$ (let us recall that $q<0$)
of $u_N$ and its other poles inside $\mathbb{D}_\varrho$ are simple and are at $
\frac{\varrho^{2k}}{\overline{x_i}}$ for $1\leq k \leq N$. We then use the variation of argument principle (see for example 
\cite{Stein} p.90) to obtain that 
\begin{equation}
\int_{C_\varrho} \frac{u'_N(z)}{u_N(z)}dz=-2i\pi |q|.
\end{equation}
But we have that $u_N$ converges uniformly to $u$ near $C_\varrho$. Thus $u'_N$ converges 
uniformly to $u'$ near $C_\varrho$. We can pass to the limit and obtain
\begin{eqnarray}
\frac{1}{2i\pi}\int_{C_\varrho}\frac{u'(z)}{u(z)}dz &=&\frac{1}{2\pi}\int_{C_\varrho} u\wedge 
\partial_\tau u \nonumber \\
&=&\deg(u,C_\varrho) \nonumber \\
&=& q. \nonumber
\end{eqnarray}
This concludes the proof of the proposition because we showed that every holomorphic functions in $\I$ is minimizing in its class $\I_{p,q}$ for $p>0>q$. (c.f. lemma \ref{minholo}).
\end{proof}

Several remarks are in order concerning this proof. We want to indicate how we derived the formula \eqref{holomorphicsol} for holomorphic solutions in $\I$. If we assume that such a function $u \in \I_{p,q}$ exists. Then $u$ has $p-q$ zeros counted with their multiplicities (c.f. lemma \ref{numberofzeros2}).  We denote them by $x_1,x_2,...,x_{p-q}$. We then set $G=\ln |u|$ , since $u$ is holomorphic in $A$, $G$ satisfies
\begin{equation}\label{linear}
\left\{
\begin{array}{rcll}
\Delta G & =&2\pi \displaystyle{\sum_{i=1}^{p-q} \delta_{x_i} }, & \text{in} \ A, \\
G&=&0, & \text{on} \   \mathbb{S}^1, \\
G&=&0, & \text{on} \ C_\varrho
\end{array}
\right.
\end{equation}
\noindent where $\delta_{x_i}$ denotes the Dirac measure at $x_i$. Since the problem \eqref{linear} is linear, to solve it it suffices to find a solution of 
\begin{equation}\label{linear2}
\left\{
\begin{array}{rcll}
\Delta G & =&2\pi \delta_{x_0} , & \text{in} \ A, \\
G&=&0, & \text{on} \   \mathbb{S}^1, \\
G&=&0, & \text{on} \ C_\varrho
\end{array}
\right.
\end{equation}

\noindent where $x_0$ is a point in $A$. By definition a solution of \eqref{linear2} is called a Green's function. Hence 
we must find a Green function in the ring $A$. This Green's function is unique for each $x_0 \in A$. We can 
find the construction of such a Green's function in the book of Courant and Hilbert \cite{courant1989methods} 
p.388-389. In fact Courant and Hilbert affirm that in order to find a 
Green's function in a ring it suffices to find a holomorphic function, with modulus constant equal to one on the 
boundaries with a simple zero in $x_0$.  However the holomorphic function that they give is not a function but a 
\textit{multivalued} function. We do not want to discuss here the notion of multivalued function, we just call multivalued a function where appears a term of the form $z^\alpha$ with $\alpha$ non integer. Note that a holomorphic function in $\I$ with a simple zero should be either in 
$\I_{1,0}$ or in $\I_{0,-1}$ but lemma \ref{obstruction} says that there is no holomorphic function in such class. \\
If $x_0\in \mathbb{C}$ then thanks to similar computations done in proposition \ref{holomorphicsolution} it can be shown that the solution of \eqref{linear2} is 
\begin{equation}\nonumber
G_{x_0} =\ln |F_{x_0}|
\end{equation}
\noindent where

\begin{equation}\nonumber
F_{x_0}=|x_0|z^{\frac{-\ln|x_0|}{\ln\varrho}}\frac{1-\frac{z}{x_0}}{1-z\overline{x_0}}\prod_{k=1}^{+\infty} \frac{(1-\frac{\varrho^{2k}z}{x_0})(1-\frac{\varrho^{2k}x_0}{z})}{(1-\frac{\varrho^{2k}}{z\overline{x_0}})(1-\varrho^{2k}z\overline{x_0})}
\end{equation}

Thus a solution of \eqref{linear} is given by 

\begin{equation}\nonumber
G=\ln|F_{x_1}F_{x_2}...F_{x_{p-q}}|
\end{equation}

\noindent and we can show that $F_{x_1}F_{x_2}...F_{x_{p-q}}$ is an holomorphic (single-valued) function  if and only if $\displaystyle{\sum_{i=1}^{p-q} \frac{\ln|x_i|}{\ln\varrho}}$ is an integer. This is exactly the condition \eqref{Conditionpointsholo}. Thus this geometric condition on the positions of the $x_i$ are sufficient and necessary for an holomorphic function in $\I$ with prescribed zeros at the $x_i$'s to exist.

\section{The case $c \neq 0$ : properties of radial solutions}

\subsection{Non existence of solutions  with $c\neq 0$ in $\I_{p,q}$ if $p\neq q$}

\begin{proposition}\label{nopq}
Let $u$ be a solution of \eqref{semi-stiff 1} in $\I_{p,q}$ with $c\neq 0$, where $c$ is the constant in the Hopf differential \eqref{Hopf} then $p=q$.
\end{proposition}

\begin{proof}
Let $u$ be a solution of \eqref{semi-stiff 1} with $c \neq 0$, we have thanks to lemma \ref{Hopfcst}
$$r^2|\partial_ru|^2- |\partial_\theta u|^2=c.$$

\begin{itemize}
\item[*] If $c>0$ then we can see that $\partial_ru$ does not vanish in $A$. We can then set $T:=\frac{\partial_ru}{|\partial_ru|}: A \rightarrow \mathbb{S}^1$.
The map $T$ takes all its values in $\mathbb{S}^1$ then it has same degree on the two boundaries of $A$, i.e. on $\mathbb{S}^1$ and $C_\varrho$. This comes from the fact that $\partial_xT\wedge \partial_yT=0$ because $T$ is $\mathbb{S}^1$-valued and from the following formula which is true for all $T$ in $\I$ (c.f. proposition \ref{diffdegree})
\begin{equation}
\int_A \partial_xT \wedge \partial_yT = \deg(T,\mathbb{S}^1)-\deg(T,C_\varrho).
\end{equation}

But we claim that $\deg(T,\mathbb{S}^1)=\deg(u,\mathbb{S}^1)$ and $\deg(T,C_\varrho)=\deg(u,C_\varrho).$
\\
Indeed on $\partial A$ we have 
$$T \wedge u =0 $$
because $u \wedge \partial_r u =0$ on $\partial A$ thanks to \eqref{semi-stiff 1}. Then on $\mathbb{S}^1$ there exists a real function $\lambda_1$ such that $\partial_ru (e^{i\theta})=\lambda_1(\theta) u(e^{i\theta})$.

 We can write $\lambda_1=\frac{\partial_r u}{u}$ because $|u|=1$ on $\mathbb{S}^1$, thus $\lambda_1$ is continuous and does not vanish thus $\lambda_1$ has constant sign. Then 
$T=u$ on $\mathbb{S}^1$ or $T=-u$ on $\mathbb{S}^1$ and hence $\deg(T,\mathbb{S}^1)=\deg(u,\mathbb{S}^1)$.
On the other hand for the same reason we also have on $C_\varrho$ the existence of $\lambda_2$ such that $\partial_ru (\varrho e^{i\theta})=\lambda_2(\theta) u(\varrho e^{i\theta})$ and the same argument as before implies that $T=u$ on $C_\varrho$ or $T=-u$ on $C_\varrho$ thus $\deg(T,C_\varrho)=\deg(u,C_\varrho)$ and finally

$$\deg(u,\mathbb{S}^1)=\deg(u,C_\varrho).$$

\item[*] If $c<0$ then this time $\partial_\theta u$ does not vanish in $A$ and we can consider the map $R:=\frac{\partial_\theta u}{|\partial_\theta u|}$. We can conclude by the same argument except that this time there exists a function $\lambda$ defined on $\partial A$ such that $\partial_\theta u= i\lambda u$ on $\partial A$. We have $\lambda= \frac{\partial_\theta u}{u}$ on $\mathbb{S}^1$ then $\lambda$ is continuous and does not vanish on $\mathbb{S}^1$. This implies that $R=iu$ or $R=-iu$ on $\mathbb{S}^1$ and we can deduce that $\deg(R,\mathbb{S}^1)=\deg(u,\mathbb{S}^1)$. The same argument on $C_\varrho$ provides us with $\deg(R,C_\varrho)=\deg(u,C_\varrho)$ and then

$$\deg(u,\mathbb{S}^1)=\deg(u,C_\varrho).$$

\end{itemize}
\end{proof}

\begin{corollary}\label{nop0}
The following hold: 
\begin{itemize}
\item[1)] For all $p\in \mathbb{N}^*$ there is no solution of \eqref{semi-stiff 1} in $\I_{p,0}$.
\item[2)] Let $p>0>q$ two integers, there is no non-holomorphic solution of \eqref{semi-stiff 1} in $\I_{p,q}$ and hence no non-minimizing solution of \eqref{semi-stiff 1} in $\I_{p,q}$.
\end{itemize}
\end{corollary}

\begin{proof}
Proposition \ref{nopq} implies that there is no solution in $\I_{p,0}$ nor in $\I_{p,q}$ for $p>0>q$ integers with $c \neq 0$ (where $c$ is the constant in the Hopf differential $\H_u(z)=\frac{c}{z^2}$). Hence 
\begin{itemize}
\item[1)] if there exists a solution of  \eqref{semi-stiff 1} in $\I_{p,0}$ for $p>0$ then its Hopf differential satisfies $c=0$ and $\H_u(z) =0$ thus it must be holomorphic or antiholomorphic. However lemma \ref{obstruction} (and its straightforward adaptation to the antiholomorphic case) implies that there is no holomorphic functions in $\I_{p,0}$.
\item[2)] Again if there exists a solution of \eqref{semi-stiff 1} in $\I_{p,q}$ with $p>0>q$ then thanks that what precedes it must be holomorphic or antiholomorphic. But we have seen in lemma \ref{minholo} that holomorphic solutions (or antiholomorphic solutions) minimize the Dirichlet energy in $\I_{p,q}$ for $p>0>q$.
\end{itemize}
\end{proof}

The corollary \ref{nop0} and the results obtained in the previous section imply  \textbf{ theorem \ref{theorem1}}.
\begin{corollary}\label{nopq2}

There is no solution of  \eqref{semi-stiff 1} in $\I_{p,q}$ with $p>q > 0$, in particular there is no minimizer of $E$ in $\I_{p,q}$ with $p>q>0$.
\end{corollary}

\begin{proof}
This corollary is obtained by combining the preceding proposition \ref{nopq} and lemma \ref{obstruction}. Indeed we have seen that there is no solution of \eqref{semi-stiff 1} with $c\neq 0$ in $\I_{p,q}$ if $p>q>0$. But if there is a solution with $c=0$ then it must be conformal and because of lemma \ref{obstruction} we must have $p$ and $q$ of opposite signs.
\end{proof}

In order to complete the proof of theorem \ref{theorem3} it remains to show

\begin{proposition}\label{theorem32}
Let $p>q>0$ we have $m(p,q)= m(q,q) + \pi(p-q)$.
\end{proposition}

\begin{proof}
Thanks to the previous corollary \ref{nopq2} we know that $m(p,q)$ is not attained. Let $(u_n)$ be a minimizing sequence for $m(p,q)$ then, up to a subsequence $u_n$ converges weakly in $H^1$ to some $u \in \I_{r,s}$. Thanks to lemma \ref{fall} we have $r \leq p$ and $s \leq q$, but applying lemma \ref{behavior1} we also have that $u$ minimizes the Dirichlet energy in its class, thus $u$ is a solution of \eqref{semi-stiff 1} and 
$$E(u)=m(r,s)=m(p,q) -\pi[(p-r) +(q-s)].$$

We are going to prove that  $u$ belongs to some $\I_{d,d}$ for some $ d \geq 1$. Indeed $u$ can not belong to some $\I_{p',q'}$ for $p'\neq q'$  and $p',q'>0$ because of corollary \ref{nopq2}. Furthermore because $u$ is a solution of \eqref{semi-stiff 1}, we have several possibilities : either $u \in \I_{d,d}$ for some $d \in \mathbb{Z}$ or  $u \in \I_{r,s}$ with $r>0>s$ and $u$ is holomorphic or $u \in \I_{r,s}$ with $s>0>r$ and $u$ is antiholomorphic. We claim that the two last cases do not occur. Indeed if $u \in \I_{r,s}$ with $r>0>s$ and $u$ is holomorphic for example. We then have 

\begin{eqnarray}
E(u) &=& \pi(r+|s|) =  \lim_{n\rightarrow +\infty } E(u_n) -\pi(p-r +q+|s|) \nonumber \\
&=& m(p,q) -\pi(p+q) +\pi(r-|s|) \label{ineq}
\end{eqnarray}

The last equalities are obtained by using lemma \ref{behavior1} and \ref{minholo}. However because we always have $m(1,1)=2\pi\frac{1-\varrho}{1+\varrho} < 2\pi$ (c.f. theorem \ref{BerlyandMironescu}) we obtain that $m(p,q) < 2\pi +\pi( p-1 +q-1)$ for $p>q>0$ by lemma \ref{Comparison1} and then 
$$m(p,q) < \pi(p+q)$$
which implies, using \eqref{ineq} that $\pi(r+|s|) < \pi(r-|s|)$ which is a contradiction. The reasoning for $u$ antiholomorphic is the same. \\
Thus we obtain that $u \in \I_{d,d}$ for some $d \in \mathbb{Z}$ and 

\begin{equation}\label{eqfall}
E(u)=m(d,d) = \lim_{n \rightarrow +\infty} E(u_n) -\pi(p-d +q-d) =m(p,q) -\pi(p-d+q-d).
\end{equation}

Hence we deduce that $d \geq 1$. Indeed because  $m(p,q) < \pi(p+q)$ if $d \leq 0$ then $E(u) <0$ which is a contradiction. Thus thanks to what precedes we have that 
\begin{equation}\label{fall2}
m(p,q)=m(d,d) +\pi(p-d+q-d). 
\end{equation}

\noindent for some $1 \leq d \leq q$. If $d=q$ the proposition is proved, but it can occur that $d<q$, in this case we claim that $$m(q,q)=m(d,d)+ \pi(q-d).$$

Indeed because of lemma \ref{Comparison1} we have $m(q,q) \leq m(d,d)+ 2\pi(q-d)$. But if 
$m(q,q) < m(d,d)+ 2\pi(q-d)$ applying lemma \ref{Comparison1} again we obtain $m(p,q) \leq 
m(q,q)+\pi(p-q) < m(d,d) +2\pi(q-d) +\pi(p-q) < m(d,d) +\pi(p-d+q-d)$ which is a contradiction  
with \eqref{fall2}. This proves that $m(p,q)=m(q,q)+\pi(p-q)$.
\end{proof}

Corollary \ref{nopq2} and proposition \ref{theorem32} proves \textbf{theorem \ref{theorem3}}. 
In order to conclude this subsection we prove that all solutions $u$ of \eqref{semi-stiff 1} which satisfy that $(\text{Re}(u_{|\partial A}), \text{Im}(u_{|\partial A}))$ are Steklov eigenfunctions are radially symmetric. 

\begin{proposition}\label{linkwithSteklov}
Let $u$ be a solution of \eqref{semi-stiff 1} such that there exists a constant $\sigma$ such that $\partial_\nu u =\sigma u$ on $\partial A$. Then $u$ is radial i.e. $u=\alpha u_p:=\alpha  \frac{1}{1+\varrho^p}(r^p+\frac{\varrho^p}{r^p})e^{ip\theta}$ or $u=\alpha \tilde{u}_p:=\alpha \frac{1}{1-\varrho^p}(r^p-\frac{\varrho^p}{r^p})e^{ip\theta}$ for some constant of modulus one $\alpha$ and some $p\in \mathbb{Z}^*$.
\end{proposition}

\begin{proof}
Let $u$ be a solution of \eqref{semi-stiff 1}, assume furthermore that  $(\text{Re}(u_{|\partial A}), \text{Im}(u_{|\partial A}))$ are Steklov eigenfunctions. We then have 
$$\partial_\nu u =\sigma u, \ \text{on} \ \partial A$$
for some $\sigma \in \R$.
Using the fact that $|u|=1$ on $\partial A$ we find that $|\partial_\nu u| =|\sigma|$ is constant on $\partial A$. We know use \eqref{constHopf} to obtain that 
\begin{eqnarray}
|\partial_\theta u|^2=-c+\sigma^2, \ \text{on} \ \mathbb{S}^1 \nonumber \\
|\partial_\theta u|^2=-c+\varrho^2 \sigma^2 \text{on} \ C_\varrho. \nonumber
\end{eqnarray}

Near the boundaries of $A$ we can write $u=|u|e^{i\varphi}$. Since $|u|=1$ on $\partial A$ we find that $|\partial_\theta u| =|\partial_\theta \varphi|$. Hence we deduce that $|\partial_\theta \varphi |$ is constant on each boundary of $A$. Thus we obtain that $\varphi(\theta) = a\theta +b$ on $\mathbb{S}^1$ and  $\varphi(\theta) = a'\theta +b'$ on $C_\varrho$ with $(a',b') \in \R$. Because $\deg(u,\mathbb{S}^1)=\frac{1}{2\pi}\int_0^{2\pi} \partial_\theta \varphi d\theta$ and a similar relation is true on $C_ \varrho$ we find that $a,a'$ are integers. Hence $u$ is the harmonic extension of $g=\alpha e^{ip\theta}$ on $\mathbb{S}^1$ ($\alpha\in \mathbb{S}^1$) and $g=\beta e^{iq\theta}$ on $C_\varrho$ ($\beta \in \mathbb{S}^1$) for some $(p,q) \in \mathbb{Z}^2$. We can compute explicitly these harmonic extensions (for example using Fourier coefficients) and a simple but tedious computation (see appendix B) shows that $u$ is a solution if and only if $u=\alpha u_p$ or $u=\alpha \tilde{u}_p$ for some $p\in \mathbb{Z}^2$ and some $\alpha \in \mathbb{S}^1$.
\end{proof}
\subsection{The case $c<0$ : geometry of the catenoid}

From now on, because of the previous corollary \ref{nopq2} we are looking for critical points 
of $E$ in $\I_{p,p}$ for $p>0$ (the case $p<0$ being obtained by complex conjugation). In this 
subsection we are interested in solutions of \eqref{semi-stiff 1} with $c<0$ (where $c$ is the 
constant in the Hopf differential). This case is of particular importance because if $u$ is a 
minimizer of $E$ in $\I_{p,p}$ then its Hopf differential satisfies $c<0$ (c.f. proposition 
\ref{1symetry} below). We will see that in this case a solution $u$ can be lifted to a minimal 
surface bounded by two $p$-coverings of circles in parallel planes. In degree one case the theorem of Shiffman 
gives that this surface is a portion of catenoid and then all solutions of \eqref{semi-stiff 1} 
in $\I_{1,1}$ with $c<0$ are radial. However if the degree  $p$  of the solution is greater 
than $2$ then the theorem of Shiffman does not apply and we will construct non radial solutions 
of \eqref{semi-stiff 1} in the next section.

\begin{proposition}\label{1symetry}
Let $u$ be a minimizer of $E$ in $\I_{p,p}$ then $\frac{\partial u}{\partial r}=0$ on $C_{\sqrt{\varrho}}$ and $c<0$.
\end{proposition}
\begin{proof}
If $\frac{\partial u}{\partial r}=0$ on $C_{\sqrt{\varrho}}$ we have on $C_{\sqrt{\varrho}}$: $c=0-|\partial_\theta u |^2$
where $c$ is the constant in the Hopf differential \eqref{Hopfcst}, hence $c \leq 0$ in $A$. However if $c=0$ then $u$ is holomorphic or antiholomorphic and then $u \in \I_{r,s}$ with $r>0>s $ or $r<0<s$  thanks to lemma \ref{obstruction} and this is a contradiction with $u \in \I_{p,p}$. The fact that $\frac{\partial u}{\partial r}=0$ on $C_{\sqrt{\varrho}}$ is a consequence of the following lemma \ref{reduction1}.

We set 
$$\F_p=\{u\in H^1(A_{\sqrt{\varrho}},\C); |u|=1 \ \text{a.e.  on} \ \mathbb{S}^1, \ \deg(u,\mathbb{S}^1)=p\}.$$
$$E_r(u):=\int_{A_r} |\nabla u|^2$$

\begin{lemma}\label{reduction1}
Assume that there exists  $u \in \I_{p,p}$ such that $E_{\varrho}(u)=\min \{E_\varrho(v) ;v \in \I_{p,p}\}$ then 
$E_{\sqrt{\varrho}}(u)=\min \{E_{\sqrt{\varrho}}(v) ; v \in \F_p \}$.\\ Conversely if there exists $u_0$ such that 
$E_{\sqrt{\varrho}}(u_0)=\min \{E_{\sqrt{\varrho}}(v) ; v \in \F_p \}$ then there exists $u$ such that $u$ is a minimizer of $E$ in $\I_{p,p}$.
\end{lemma}

\begin{remark} Thanks to this lemma, in order to solve the problem \eqref{semi-stiff 1} we can consider the simpler problem of minimizing $E$ in $\F_p$. This problem is simpler because we prescribe constraints  only on one boundary. The condition obtained on the boundary $C_{\sqrt{\varrho}}$ is a Neumann homogeneous condition: $\frac{\partial u}{\partial \nu}=0$ on $C_{\sqrt{\varrho}}.$ This lemma comes from an idea of I.Shafrir and can be found in an other setting in the article of E.Sandier \cite{sandier1993symmetry}.
\end{remark}

Let $u \in \I_{p,p}$ we set:

$$u_1(z)= \begin{cases}   u(z) & \text{if} \ \sqrt{\varrho}<|z|<1, \\
u(\frac{\varrho}{\overline{z}}) & \text{if} \ \varrho<|z|<\sqrt{\varrho}.
\end{cases}$$

$$u_2(z)= \begin{cases} u(\frac{\varrho}{\overline{z}}) & \text{if} \ \sqrt{\varrho}<|z|<1, \\
u(z) & \text{if} \ \varrho<|z|<\sqrt{\varrho}.
\end{cases}$$

We have that $u_1,u_2 \in {H}^1(A,\C)$ ($A=A_{\varrho}=\{z\in \C; \varrho < |z| <1 \}$). For example for $u_1$ we can see that $u_1 \in L^2(A, 
\C)$ thanks to the change of variable formula and the fact that $u\in L^2(A,\C)$. Furthermore 
because $z=\frac{\varrho}{\overline{z}}$ on $C_{\sqrt{\varrho}}$ we have that $u(z)=u(\frac{\varrho}
{\overline{z}})$ on $C_{\sqrt{\varrho}} $ and then $\nabla u_1 =\nabla u \mathbb1 _{A_{\sqrt{\varrho}}} + 
\nabla [u(\frac{\varrho}{\overline{z}})] \mathbb{1}_{A\setminus{A_{\sqrt{\varrho}}}}$. Then because 
the map $z \mapsto \frac{\varrho}{\overline{z}}$ is conformal we obtain
$$E_\varrho(u_1)=2\int_{A_{\sqrt{\varrho}}}|\nabla u|^2 $$
$$E_\varrho(u_2)=2\int_{A_\varrho \setminus A_{\sqrt{\varrho}}}|\nabla u|^2 $$

Hence $2E_\varrho(u)=E_\varrho(u_1)+E_\varrho(u_2)$. Observe that if $u\in \I_{p,p}$ then $u_1\in \I_{p,p}$ and $u_2\in \I_{p,p}$. Let us assume that $u$ is a minimizer of $E_\varrho$ in $\I_{p,p}$ then $E_\varrho(u)=E_\varrho(u_1)=E_\varrho(u_2)$, indeed if it were not the case either $u_1$ or $u_2$ would satisfy $E_\varrho(u_i)<E_\varrho(u)$, thus $u_1$ and $u_2$ are also minimizers. Now let $v\in \F_p$ we have 
$$E_{\sqrt{\varrho}}(v) \geq E_{{\sqrt{\varrho}}}(u)$$

\noindent Indeed define

$$ 
v_1(z)= \begin{cases} v(z) & \text{if} \ \sqrt{\varrho}<|z|<1, \\
v(\frac{\varrho}{\overline{z}}) &  \text{if} \ \varrho<|z|<\sqrt{\varrho}.
\end{cases}
$$

\noindent We have $v_1\in \I_{p,p}$ and because $u_1$ is a minimizer of $E_\varrho$ we also have
\begin{eqnarray}
E_\varrho(v_1) & \geq & E_\varrho(u_1) \nonumber\\
& \geq & 2E_{\sqrt{\varrho}}(u) \nonumber
\end{eqnarray}
\noindent and
\begin{eqnarray}
2E_{\sqrt{\varrho}}(v) & \geq & 2E_{\sqrt{\varrho}}(u) \nonumber \\
E_{\sqrt{\varrho}}(v) & \geq & E_{\sqrt{\varrho}}(u) \nonumber
\end{eqnarray}

\noindent then we can deduce that the restriction of $u$ to $A_{\sqrt{\varrho}}$ minimizes $E_{\sqrt{\varrho}}$ in $\F_p$.

Conversely if $v\in \F_p$ minimizes $E_{\sqrt{\varrho}}$ it is not difficult to see, thanks to the previous computation that 
$$ 
v_1(z)= \begin{cases} v(z) & \text{if} \ \sqrt{\varrho}<|z|<1, \\
v(\frac{\varrho}{\overline{z}}) &  \text{if} \ \varrho<|z|<\sqrt{\varrho}.
\end{cases}$$

\noindent minimizes $E_\varrho$ in $\I_{p,p}$.
This proves the lemma and then the proposition \ref{1symetry}.
\end{proof}

\begin{proposition}\label{buildingminsurf1}
 Let $u$ be a solution of \eqref{semi-stiff 1} with $\H_u (z)=\frac{-|c|}{z^2}$. If we set 
$$ h(x,y)=\sqrt{|c|}\ln(x²+y^2), \ \ \ for \ (x,y)\in A $$
then 
\begin{equation}
\begin{array}{rcll}
X:A & \rightarrow & \R³ \\
(x,y) & \mapsto & (u(x,y), h(x,y))
\end{array}
\end{equation}
is a conformal immersion which parametrizes a minimal surface.
\end{proposition}

\begin{proof}
This is an application of proposition \ref{lifting}. 
\end{proof}

\begin{proposition}\label{planesymmetry}
Let $p>0$ be an integer. Let $u$ be a minimizer of the Dirichlet energy $E$ in $\I_{p,p}$. Then the minimal 
surface obtained by the process of the previous proposition \ref{buildingminsurf1} is symmetric with respect to 
the plane $\{ z =\sqrt{|c|}\ln(\varrho) \}$ in $\R^3$. 
\end{proposition}

\begin{proof}
This is a consequence of proposition \ref{1symetry}. Indeed  if $u$ is a minimizer of $E$ in $\I_{p,p}$ then 
$$\partial_ru =0 \ \text{on} \ C_{\sqrt{\varrho}}$$

\noindent and $\H_u(z)=\frac{c}{z^2}$ with $c<0$. This imply that the minimal surface given by $X=(u,h):A \rightarrow \
\R^3$,  where $h=2\sqrt{|c|}\ln(r)$ (with $r=\sqrt{x^2+y^2}$) intersects the plane $\{z =\sqrt{|c|}\ln(\varrho) 
\}$ perpendicularly. Indeed the  intersection between the minimal surface and this plane is the plane curve 
$X(\sqrt{\varrho} e^{i\theta})$, $\theta \in [0,2\pi[$ and the normal of the surface on this curve is given by 
$$N=\frac{\partial_x X\wedge \partial_y X}{|\partial_x X\wedge \partial_y X|}= \frac{\partial_r X \wedge 
\partial_\theta X}{|\partial_r X \wedge \partial_\theta X |}$$ taken at $(\sqrt{\varrho},\theta )$ for $\theta \in 
[0,2\pi[$. But because of the fact that $\partial_ru =0$ on $C_{\sqrt{\varrho}}$ we obtain that $N$ is horizontal on 
the curve  $X(\sqrt{\varrho} e^{i\theta})$, $\theta \in [0,2\pi[$ and hence the minimal surface parametrized by $X$ 
intersects the plane $\{z =\sqrt{|c|}\ln(\varrho) \}$ perpendicularly. Now we use the following classical symmetry 
result due to H.A. Schwarz (see \cite{minimalsurfaces} p.128) if a minimal surface intersects some plane $P$ 
perpendicularly, then $P$ is a plane of symmetry of the surface. 
\end{proof}

In the case $\H_u(z)=\frac{c}{z^2}$ ($c<0$) the minimal surface obtained by corollary \ref{lifting} is a minimal surface bounded by two $p$-coverings of  circles in parallel planes. This is due to the fact that $|u|= 1$ on $\partial A$ and $h=-\ln r$. In fact such a minimal surface gives rise to a solution of \eqref{semi-stiff 1} with $c<0$. Thus  the problem of finding solution of \eqref{semi-stiff 1} with $c<0$ and the problem of finding minimal surfaces bounded by two $p$-coverings of circles in parallel planes are \textbf{completely equivalent}.

\begin{proposition}
Let $X=(u,h):A \rightarrow \R²$ be a conformal parametrization of a doubly connected minimal surface bounded by two $p$-coverings of  circles in parallel planes then 
$$h=a\ln r +b, \ \ with \ (a,b) \in \R^2 $$
and $u$ satisfies 

\begin{equation}
\left\{
\begin{array}{rclll}
\Delta u & =& 0, & \text{in} \ A, \\
|u|&=&C_1,& \text{a.e. on} \ \mathbb{S}^1 \\
|u|&=&C_2, & \text{a.e. on} \ C_\varrho \\
u\wedge \partial_\nu u & =& 0, & \text{a.e. on} \ \partial A.
\end{array}
\right.
\end{equation}
with $C_1,C_2$ two real constants.
\end{proposition}

\begin{proof}
The height function $h$ is constant on each of the circles bounding the annular ring, $h$ being harmonic we must have $h=a \ln r +b$. Indeed $h$ solve an equation of the following form 

\begin{equation}
\left\{
\begin{array}{rclll}
\Delta h &=&0, & \text{in} \ A \\
h&=& c_1, & \text{on} \ \mathbb{S}^1 \\
h&=& c_2, & \text{on} \ C_\varrho
\end{array}
\right.
\end{equation}
This is a Dirichlet problem for the Laplacian and the solution is unique, of the form $h=a\ln r +b$. 
Now we apply proposition \ref{confandhopf} to find that 
$$\H_u(z)=-(\partial_zh)^2$$

But $\partial_zh=\frac{1}{2}(\partial_rh-i\frac{1}{r}\partial_\theta h)e^{-i\theta}=\frac{b/2}{r}e^{-i\theta}$
thus $$\H_u(z)=\frac{c}{z^2}$$ 
with $c=-b^2/4<0$. 
Now recall that 
$$z^2\H_u(z)= r²|\partial_ru|²-|\partial_\theta u|²-2ir \langle \partial_ru,\partial_\theta u \rangle$$ 
then we can deduce that 
\begin{equation}\label{orthoradial2}
\langle \partial_ru,\partial_\theta u\rangle=0 \ \text{in} \ A. 
\end{equation}

Moreover we have  that $|u|=C_1$ on $\mathbb{S}^1$ and $|u|=C_2$ on $C_\varrho$ because the surface is bounded by two circles in parallel planes. This implies that $\langle u,\partial_\theta u\rangle=0$ on $\partial A$. We can then conclude from this and from \eqref{orthoradial2} that 
$$u\wedge \partial _\nu u =0, \ \text{on} \ \partial A.$$ 
\end{proof}

Thanks to a theorem of M.Shiffman and  the equivalence between the problem of finding solutions of \eqref{semi-stiff 1} with $c<0$ and finding minimal surfaces bounded by two $p$-covering of circles in parallel planes we can prove that every solution of \eqref{semi-stiff 1} in $\I_{1,1,}$ with $c<0$ is radial.

\begin{theorem}[\cite{Shiffman}]
Let $S$ be a compact minimal surface  in $\R^3$ bounded by two plane curves $\Gamma_1, \Gamma_2$ lying in parallel planes. If $\Gamma_1, \Gamma_2$ are circles then the intersection of $S$ by a plane parallel to the planes of $\Gamma_1,\Gamma_2$ is again a circle. If the two circles $\Gamma_1$ and $\Gamma_2$ have a common axis of symmetry then the minimal surface $S$ is a portion of catenoid.
\end{theorem}

\begin{theorem}
Let $u$ be any (non necessary minimizing) solution of \eqref{semi-stiff 1} in $\I_{1,1}$ with $c<0$ then there exists $\alpha \in \mathbb{S}^1$ such that

\begin{equation}
u(z)= \frac{\alpha}{1+\varrho}(r+\frac{\varrho}{r})e^{i \theta}
\end{equation}
\end{theorem} 

Let us recall that L.Berlyand and P.Mironescu proved in \cite{unpublished} that the function $u_1(z)=\frac{1}{1+\varrho}(r+\frac{\varrho}{r})e^{i \theta}$ is the only (modulo rotations) minimizer of $E$ in $\I_{1,1}$.

If $p>1$, the theorem of Shiffman and its proofs do not apply so we can not conclude that every solution of \eqref{semi-stiff 1} with $c<0$ is radial. However we can use the same method as L.Berlyand and D.Golovaty in \cite{uniqueness} to prove that if the annulus is thin enough then we have existence and uniqueness (modulo rotations) of a minimizer of $E$ in $\I_{p,p}$.

\begin{theorem}\label{unicityrad}
Let $p \geq 2$. There exists $\varrho_p<1$ such that if $\varrho> \varrho_p$ then $E$ has a unique (up to an arbitrary rotation), radially symmetric minimizer
\begin{equation}
u_p(z)= \frac{\alpha}{1+\varrho^p}(r^p+\frac{\varrho^p}{r^p})e^{ip\theta}.
\end{equation}
\end{theorem}

\begin{proof}
In order to prove this theorem we can follow step by step the proof of an analogous result for the Ginzburg-Landau energy of L.Berlyand and D.Golovaty in \cite{uniqueness} taking $\e=+\infty$.
\end{proof}

In fact, contrarily to the case $p=1$, if $p>1$ and if the annulus is too thick then the radial solution is not minimizing anymore. Such similar phenomenon occurs in the study of the Ginzburg-Landau energy, but for the G.L energy even in degree one case the existence of minimizer depends on the capacity of the domain c.f. \cite{Nonexistence}, \cite{unpublished}, \cite{Capacity}. 

\begin{theorem} \label{nonminrad}
Let $p\geq 2$. There exists $\varrho_p'\leq \varrho_p$ such that if $\varrho < \varrho_p'$ then \\
$u_p(z)= \frac{\alpha}{1+\varrho^p}(r^p+\frac{\varrho^p}{r^p})e^{ip\theta}$ is not a minimizer of $E$ in $\I_{p,p}$.
\end{theorem}

\begin{proof}
Using lemma \ref{Comparison1} we have that for $p \geq 2$ 

\begin{equation}
m(p,p)\leq m(1,1) +2\pi(p-1).
\end{equation}

\noindent We know that $u_1(z)=\frac{1}{1+\varrho}(r+\frac{\varrho}{r})e^{i \theta}$ is a minimizer of $E$ in $\I_{1,1}$ thus
$$m(1,1)=E(u_1)=2\pi\frac{1-\varrho}{1+\varrho}.$$

\noindent Let $u_p(z)= \frac{1}{1+\varrho^p}(r^p+\frac{\varrho^p}{r^p})e^{ip\theta}$ then a direct computation leads to 
$E(u_p)= 2\pi p\frac{1-\varrho^p}{1+\varrho^p}.$

\noindent Hence if we have 
\begin{equation}\label{inequrad}
E(u_p)>E(u_1)+2\pi(p-1)
\end{equation}
then $u_p$ can not be a minimizer of $E$ in $\I_{p,p}$. Let us examine the condition \eqref{inequrad}:

\begin{eqnarray}
 \eqref{inequrad} & \Leftrightarrow & 2\pi p \frac{1-\varrho^p}{1+\varrho^p} > 2\pi \frac{1-\varrho}{1+\varrho}+2\pi(p-1) \nonumber \\
 & \Leftrightarrow & p(1-\frac{2\varrho^p}{1+\varrho^p}) >(1-\frac{2\varrho}{1+\varrho}) + p-1 \nonumber \\
 & \Leftrightarrow & \frac{2\varrho}{1+\varrho}> \frac{2p\varrho^p}{1+\varrho^p} \nonumber \\
 &  \Leftrightarrow & 1+\varrho^p >p\varrho^{p-1}(1+\varrho) \nonumber \\
 & \Leftrightarrow & (p-1)\varrho^p+p\varrho^{p-1}-1 <0. 
\end{eqnarray}

We set $g_p(\varrho)=(p-1)\varrho^p+p\varrho^{p-1}-1$. One can study  the function $g_p$ for $0\leq \varrho\leq 1$ and show that  

\begin{equation}\label{studyofg}
\left\{
\begin{array}{rclll}
g_p(0)&=&-1<0, \\
g_p(1) &=&2(p-1) >0, \\
g_p'(\varrho)&= & p(p-1)(\varrho^{p-1}+\varrho^{p-2})\geq0,\ \forall \ \varrho \in [0,1]. 
\end{array}
\right.
\end{equation}

Thus \eqref{studyofg} proves that there exists $\varrho_p'$ such that if $ 0<\varrho<\varrho_p'$ then the solution $u_p$ is not minimizing. The fact that $\varrho_p' \leq \varrho_p$ follows from the definition of $\varrho_p$ in theorem \ref{unicityrad}. To prove that the sequence $(\varrho'_p)_p$ is increasing with $p$ we show that $g_{p+1}(\varrho'_p)<0$ for all $p \in \mathbb{N}$, $p\geq 2$. This will prove the monotonicity of the sequence $\varrho'_p$ because $g_p$ is an increasing function on $[0,1]$ and by definition $g_p( \varrho'_p)=0$.

$$g_p(\varrho'_p)=0 \Leftrightarrow (p-1)\varrho_p^{'p}+p \varrho_p^{'p-1} =1  .$$

\noindent Hence we deduce that 
\begin{eqnarray}
g_{p+1}(\varrho'_p)&= &p \varrho_p^{'p+1}+(p+1)\varrho_p^{'p}-1 \nonumber \\
&=& p\varrho_p^{'p+1}+(p+1)\varrho_p^{'p}-(p-1)\varrho_p^{'p}-p\varrho_p^{'p-1} \nonumber \\
&=& \varrho_p^{'p-1}(p\varrho_p^{'2}+2\varrho'_p-p) \nonumber
\end{eqnarray}

Then we can study the sign of $p\varrho_p^{'2}+2\varrho'_p-p$. We need to show that $\varrho^{'p} < \frac{-1+\sqrt{1+p^2}}{p}$ to deduce that $p\varrho_p^{'2}+2\varrho'_p-p<0$. To this end we can prove that $g_p(\frac{-1+\sqrt{1+p^2}}{p})>0$. Simple but tedious computations lead to 
$$g_p(\frac{-1+\sqrt{1+p^2}}{p})=\frac{(\sqrt{1+p^2}+p)p^p-(\sqrt{1+p^2}+1)^p}{(\sqrt{1+p^2}+1)^p}.$$ One can study the sign of the numerator and obtain that 
\begin{eqnarray}
(\sqrt{1+p^2}+p)p^p-(\sqrt{1+p^2}+1)^p & > & (\sqrt{1+p^2}+1)(p^p-(\sqrt{1+p^2}+1)^{p-1}) \nonumber \\
& > &(\sqrt{1+p^2}+1) (p^p-(p+2)^{p-1}).\nonumber
\end{eqnarray}
One shows that $p^p-(p+2)^{p-1}\geq 0$ for all $p \geq 2$ by showing that $p\ln p -(p-1)\ln(p+2) \geq 0$ for $p \geq 2$ (we study the function $l(x)=x\ln x -(x-1)\ln (x+2)$). This proves the strict monotonicity of the sequence $(\varrho^{'}_{p})_p.$
  The fact that $\varrho_2' = \sqrt{2}-1$ is obtained by solving 
$g_2(\varrho'_2)=0$ which is equivalent to $\varrho_2^{'2}+2\varrho'_2-1=0$.
\end{proof}

\begin{proposition}\label{casep=2}
Let $p=2$. Let $\varrho'_2$ as in theorem \ref{nonminrad}. Let us assume that $\varrho \geq \varrho'_2$  then $m(2,2)$ is attained.
\end{proposition}

\begin{proof}
Thanks to theorem \ref{BerlyandMironescu} we know that $u_1(z)=\frac{1}{1+\varrho}(r+\frac{\varrho}{r})e^{i\theta}$ is a minimizer of $E$ in $\I_{1,1}$ for all $0<\varrho<1$ thus $m(1,1)= E(u_1)=2\pi \frac{1-\varrho}{1+\varrho}$. Besides from the definition of $\varrho'_2$ (c.f. \eqref{inequrad}) it holds that if $\varrho \geq \varrho'_2$ then $u_2(z)= \frac{1}{1+\varrho^2}(r^2+\frac{\varrho^2}{r^2})e^{i2\theta}$ satisfies that 
\begin{equation}\label{conttradiction}
E(u_2) \leq m(1,1)+2\pi.
\end{equation}

By contradiction if $m(2,2)$ is not attained for $\varrho \geq \varrho'_2$. Then applying lemma \ref{behavior1} and \ref{fall} we can see that 
\begin{equation}\label{contradictiion}
m(2,2)=m(1,1)+2\pi.
\end{equation}
Indeed let $(u^{(n)})$ be a minimizing sequence for $E$ in $\I_{2,2}$ thanks to lemma \ref{behavior1} we can deduce that $u^{(n)}$ converges weakly in $H^1$ to some $u \in \I_{r,s}$ and such that $E(u)=m(r,s)$. Furthermore lemma \ref{fall} allows us to obtain that $r<2$ and $s<2$. We claim that $r=s=1$. The argument is similar to the one of proposition \ref{theorem32}. Since  $u$ is a solution of \eqref{semi-stiff 1}, we have $u \in \I_{r,r}$ for some $r \in \mathbb{Z}$, $u \in \I_{r,s}$ with $2>r>0>s$ and $u$ is holomorphic or $u \in \I_{r,s}$ with $2>s>0>r$ and $u$ is antiholomorphic. We claim that the two last cases do not occur. Indeed if $u \in \I_{r,s}$ with $r>0>s$ and $u$ is holomorphic for example. We then have 

\begin{eqnarray}
E(u) &=& \pi(r+|s|) =  \lim_{n\rightarrow +\infty } E(u_n) -\pi(2-r +2+|s|) \nonumber \\
&=& m(2,2) -4\pi +\pi(r-|s|) \label{ineq22}
\end{eqnarray}

The last equalities are obtained by using lemma \ref{behavior1} and \ref{minholo}. However because we always have $m(1,1) < 2\pi$  we obtain that $m(2,2) < 2\pi +2\pi=4\pi$  by lemma \ref{Comparison1}. This implies, using \eqref{ineq22} that $\pi(r+|s|) < \pi(r-|s|)$ which is a contradiction. The reasoning for $u$ antiholomorphic is the same. \\
Thus we obtain that $u \in \I_{r,r}$ for some $r \in \mathbb{Z}, r <2$ and 

\begin{equation}\label{eqfall22}
E(u)=m(r,r) = \lim_{n \rightarrow +\infty} E(u_n) -\pi(2-r +2-r) =m(2,2)-4\pi +2r\pi.
\end{equation}

Hence we deduce that $d \geq 1$. Indeed since  $m(2,2) < 4\pi$, if $r \leq 0$ then $E(u) <0$ which is a contradiction. Thus we obtain that $r=s=1$ and \eqref{contradictiion} holds. But this means that 
$$E(v) > m(1,1)+ 2\pi, \ \text{for all} \ v \in \I_{2,2}$$
and this is a contradiction with \eqref{conttradiction}. Thus the proposition is proved.
\end{proof}

Note that in the previous proposition we do not know if $u_2$ is a minimizer. We also do not know if a minimizer 
is unique (up to rotation). We obtain this information only requiring the stronger condition $\varrho >\varrho_2$ with 
$\varrho_2$ as in theorem \ref{unicityrad}. Note also that we are not able to obtain similar result for the case 
$p=q \geq 3$. Indeed the important ingredient in the proof of proposition \ref{casep=2} is that we know the exact 
value of $m(1,1)$ for all $0<\varrho <1$. We do not have this information for $m(2,2)$ and hence we can not argue by 
induction. Theorems \ref{unicityrad}, \ref{nonminrad} and proposition \ref{casep=2} prove \textbf{theorem \ref{theorem2}}.

\subsection{The case $c>0$ : geometry of the helicoid}

This case is more complicated than the previous one. This is because this time the imaginary part in the integral of \eqref{Hopfdef} is not single-valued because it is precisely $\arg(z)$. Thus it is more difficult to make a clear link between the problem \eqref{semi-stiff 1} with $c>0$ and a minimal surface problem. However we can still lift a solution to the problem \eqref{semi-stiff 1} in a minimal surface if we define the surface in an appropriated domain.

\begin{proposition}\label{buildingminsurf2}
Let $u$ be a solution of \eqref{semi-stiff 1} with $\H_u(z)=\frac{c}{z^2}$ $c>0$  then 
\begin{equation}
\begin{array}{rcll}
X:[\varrho,1]\times [0,2\pi[ & \rightarrow & \R³ \\
(r,\theta) & \mapsto & (\tilde{u}(r,\theta), -2\sqrt{c}\theta)
\end{array}
\end{equation}

(where $\tilde{u}(r,\theta)=u(r\cos\theta,r\sin\theta))$ is a conformal immersion which parametrizes a minimal surface.
\end{proposition}

We also have radial solutions of problem \eqref{semi-stiff 1} with $c>0$.
\begin{proposition}
Let $p \geq 1$ then for all $\alpha \in \mathbb{S}^1$ the functions 

\begin{equation}
\tilde{u}_p(z)= \frac{\alpha}{1-\varrho^p}(r^p-\frac{\varrho^p}{r^p})e^{ip\theta}
\end{equation}

\noindent are (non-minimizing) solutions of \eqref{semi-stiff 1}. Furthermore we have 
$E(\tilde{u}_p)=2\pi p\frac{1+\varrho^p}{1-\varrho^p}> E(u_p)= 2\pi p\frac{1-\varrho^p}{1+\varrho^p}$.
\end{proposition}

The proof is a simple verification.  \\

The solution $\tilde{u}_p$ when lifted to a minimal surface with proposition \ref{buildingminsurf2} gives rise to an helicoid. It is a well-known fact that catenoid and helicoid are two minimal surfaces which are conjugates to each other. However if we denote by $S$ the minimal surface obtained by proposition \ref{buildingminsurf1} and the solution $u_p$ and by $\tilde{S}$ the minimal surface obtained by proposition \ref{buildingminsurf2} then $S$ and $\tilde{S}$ are not conjugates. Indeed if they were, in terms of isothermal parameters the height function of $S$, denoted by $h$ and the height function of $\tilde{S}$ denoted by $\tilde{h}$ would satisfy 
$$\partial_zh =i\partial_z \tilde{h}$$
because of the definition of conjugate surface (see for example \cite{minimalsurfaces} p.93 or \cite{Hauswirth}). On the other hand because of relation \eqref{conformality2} we have  $\H_u +(\partial_z h)^2=0$ and $\H_{\tilde{u}} +(\partial_z \tilde{h})^2=0$. Hence we would obtain
$$\H_u(z)=-\H_{\tilde{u}}(z).$$
But a direct computation shows that $z^2\H_u(z)=\frac{-4p^2\varrho^p}{(1+\varrho^p)^2}$ and $z^2\H_{\tilde{u}}(z)=\frac{4p^2\varrho^p}{(1-\varrho^p)^2}$.

\section{Non rotationally symmetric minimal surfaces bounded by two circles in parallel planes}
The aim of this section is to prove that there exist non radially symmetric solutions to problem 

\begin{equation}\nonumber
\left\{
\begin{array}{rcll}
\Delta u & =& 0, \ \ \text{in} \ A, \\
|u|&=&1, \ \ \text{a.e.  on} \ \partial  A, \\
u\wedge \partial_\nu u & =& 0, \ \ \text{a.e.   on} \ \partial A.
\end{array}
\right.
\end{equation}

We have seen that finding solutions to \eqref{semi-stiff 1} with $c<0$ (where $c$ is the constant in the Hopf differential) is equivalent to finding minimal surfaces bounded by two circles in parallel planes. A beautiful theorem of M.Shiffman (see \cite{Shiffman}) proves that such a surface is necessarily a portion of a catenoid if the degree of the immersion is one. However we prove in this section the existence of non rotationally symmetric minimal surfaces obtained by bifurcation of a $p$-covering ($p\geq2$) of the  catenoid. 

\begin{theorem}\label{nonsymmetricmminsurf}
There exist non-rotationally  symmetric immersed  minimal surfaces bounded by two concentric $p$-coverings of circles in parallel planes. These surfaces are symmetric with respect to reflections around the planes $P:=\{(x,y,z) \in \R^3; \ z=0\}$ and $P':=\{(x,y,z)\in \R³; \ y=0\}$.
\end{theorem}

\begin{proof}

Let us parametrize a catenoid covered $p$ times ($p \geq 2$) by 

\begin{equation} \nonumber
X(r,\theta)=(\cosh(pr)\cos(p\theta),\cosh(pr)\sin(p\theta),pr), \ \ \ \text{for} \ r\in ]-\infty,+\infty[,\ \theta \in [-\pi,\pi].
\end{equation}

This immersion provides us with a family of compact portions of catenoids. The parameter of that family is $t$ such that $r \in [-t, t]$ ($t$ is the height of the compact portion of the catenoid).
We want to parametrize this family on a fix domain so we set 

\begin{equation}
\begin{array}{rclll}
X_t(r,\theta) &: &[-1,1] \times [-\pi,\pi] & \rightarrow & \R^3 \\
 & &(r,\theta) &\mapsto & (\cosh(tpr)\cos(p\theta),\cosh(tpr)\sin(p\theta),tpr)

\end{array}
\end{equation}

\noindent Note that this parametrization  is not conformal anymore if $t\neq 1$.
Let
$$N_t =\frac{\partial_rX_t \wedge \partial_\theta X_t}{|\partial_rX_t \wedge \partial_\theta X_t|}$$
be the unit normal vector to $X_t$, 
$$N_t = \frac{1}{\cosh^2(tpr)}(\cosh(tpr)\cos(p\theta),\cosh(tpr)\sin(p\theta),\sinh(tpr)\cosh(tpr)).$$ 
 Let us fix $0< \alpha<1$. We denote by $H$ the mean curvature operator from $\mathcal{C}^{2,\alpha}$ to $C^{0,\alpha}$. We  look for minimal surfaces of the form 

\begin{equation}\nonumber
Y_t=X_t+uN_t, \ \ \ \text{with} \ u\in \mathcal{C}^{2,\alpha}_0([-1,1]\times [-\pi,\pi], \R),
\end{equation}
\noindent where $\mathcal{C}^{2,\alpha}_0([-1,1]\times [-\pi,\pi], \R)$ is the set of $\mathcal{C}^{2,\alpha}$ functions in the annulus which are null on the boundary $\partial A$ . However for $t$ fixed $Y_t$ is an immersion only if $u$ is small enough (in the $\mathcal{C}^{2,\alpha}$ norm), that is why we need to control the norm of $u$. Moreover we need this control to be independent of $t$. 

\begin{lemma}\label{immersion}
If $\|u\|_{\mathcal{C}^{2,\alpha}} < 1$ then $Y_t=X_t+uN_t$ defines an immersion for all $t\in \R^*$.
\end{lemma}

\begin{proof}
Let us drop the subscript $t$ for this proof. Differentiating $Y$ with respect to $r$ and $\theta$ we find:
\begin{eqnarray}
\partial_r Y &=& \partial _r X+u\partial_r N +\partial_r u N \nonumber \\
\partial_\theta Y & = & \partial _\theta X + u \partial_\theta N +\partial_\theta u N \nonumber
\end{eqnarray}

We denote by $E,F,G$ the coefficients of the first fundamental form and by $e,f,g$ the ones of the second fundamental form of $X_t$. We use the fact that $f=0$ (the first and second fundamental forms of $X_t$ are computed in step 1 of this section) and  we can write 
\begin{eqnarray}
\partial_rN &=&\frac{-e}{E}\partial_r X \nonumber \\
\partial_\theta N &=& \frac{-N}{G}\partial_\theta X \nonumber 
\end{eqnarray}

\begin{eqnarray}
\partial_r Y \wedge \partial_\theta Y = (1-u\frac{e}{E})(1-u\frac{N}{G}) \partial_r X \wedge \partial_\theta X + (1-u\frac{e}{E})\partial_rX \wedge N +(1-u\frac{N}{G})\partial_\theta X \wedge N \nonumber \\
= (1-u(\frac{Ge+EN}{EG})+u^2\frac{eN}{EG}) \partial_rX \wedge\partial_\theta X \nonumber \\
+(1-u\frac{e}{E})\partial_rX \wedge N +(1-u\frac{N}{G})\partial_\theta X \wedge N. \nonumber
\end{eqnarray}

\noindent Thus a sufficient condition for $Y$ to be an immersion is 
\begin{equation}\nonumber
1-u(\frac{Ge+Eg}{EG})+u^2\frac{eg}{EG} \neq 0 
\end{equation}
\noindent and because $\frac{Ge+Eg}{EG}=0$ since $X$ is a minimal surface this amounts to ask
\begin{equation}\nonumber
1-u^2 \frac{1}{\cosh^4(tpr)} \neq  0. 
\end{equation}
Hence if $\|u\|_{L^\infty} <1$ then $1-u^2 \frac{1}{\cosh^4(tpr)} \neq 0$ for all $t \in \R$ and $X+uN$ is an immersion. This result is also true if $\|u\|_{\mathcal{C}^{2,\alpha}} <1$.

\end{proof}

We are interested in minimal surfaces which are normal on domain of the catenoid. Since we have 

$${Y_t} _{|{\{-1\}\times [-\pi,\pi[}} ={X_t} _{| {\{-1\} \times [-\pi,\pi]}}$$
$${Y_t} _{|{\{1\}\times [-\pi,\pi[}} ={X_t} _{| {\{1\}\times [-\pi,\pi]}} $$

\noindent the surfaces $Y_t$ are bounded by two $p$ coverings of circles in parallel planes, more specifically $Y_t$ and $X_t$ have same boundaries. We are looking for solutions of the following problem
\begin{equation}\label{bifurcation}
F(t,u):=H(X_t+uN_t)=0,
\end{equation}
   
\noindent with $u\in V:=\{v \in \mathcal{C}^{2,\alpha}_0([-1,1]\times [-\pi,\pi], \R); \ \|v\|_{\mathcal{C}^{2,\alpha}}<1 \}$. Thus if $u$ is a solution of \eqref{bifurcation} for some $t\in \R$ then $Y_t$ is a minimal surface bounded by two coaxial circles in parallel planes. The function $u =0$ is a trivial solution for all $t \in \R$ and we are looking for bifurcations from this trivial branch of solutions.

\begin{definition}
We say that $t^*$ is a bifurcation point for $F$ (from the trivial solution) if there is a sequence $(t_n,u_n) \in \R \times V $ with $u_n \neq 0$ and $F(t_n,u_n)=0$ such that 
$$(t_n,u_n) \rightarrow (t^*,0).$$
\end{definition}
We set 

$$ \E = \{ u\in \mathcal{C}^{2,\alpha}_0([-1,1]\times [-\pi,\pi], \R); u(r,\theta)=u(-r,\theta) \ \text{and} \ u(r,-\theta)=u(r,\theta)  \}. $$
\noindent and

$$\mathcal{V}:= V \cap \E.$$

We want to find a non trivial branch of solutions $u$ of \eqref{bifurcation} in $\mathcal{V}$ by applying the Crandall-Rabinowitz theorem (see \cite{Crandall}). Let us recall the conditions which allow us to apply this theorem. We must prove that 

\begin{itemize}\label{Conditions}
\item[1)] There exists some $t^*$ and $u^*$ such that $\text{Ker} (D_uF(t^*,0))= \text{Vect}(u^*)$.
\item[2)] For this $t^*$, $\text{Im} (D_uF(t^*,0))$ is of codimension one. 
\item[3)] Letting $M:= D_{u,t}F(t^*,0)$ we have $Mu^* \notin \text{Im}(D_uF(t^*,0))$.
\end{itemize}

By definition $D_uF(t,0)$ is the Jacobi operator of the surface $X_t$ hence we must compute the Jacobi operator of the immersion $X_t$.
\begin{itemize}

\item[\textbf{Step 1}:] Computation of the Jacobi operator. \\

Here we give the computations which lead to the expression of the Jacobi operator. The formula for the Jacobi operator is 

\begin{equation}
J_tu:=D_uF(t,0) =\Delta_{X_t} u-2K_tu
\end{equation}
\noindent where $\Delta_{X_t}$ is the Laplace-Beltrami operator of the surface $X_t$ and $K_t$ is the Gauss curvature of $X_t$. Both quantities can be computed with the first and second fundamental forms. Let
\begin{eqnarray}
I&=& Edr^2+2Fdrd\theta +Gd\theta^2 \nonumber \\
II&=& edr^2+2fdrd\theta+gd\theta^2 \nonumber
\end{eqnarray}
denote respectively the first and second fundamental forms.
One has

\begin{eqnarray}
\partial_rX_t &=& pt(\sinh(tpr)\cos(p\theta),\sinh(tpr)\sin(p\theta),1) \nonumber \\
\partial_\theta X_t &=& p (-\cosh(tpr)\sin(p\theta),\cosh(tpr)\cos(p\theta),0). \nonumber
\end{eqnarray}
Thus we can write the first fundamental form of $X_t$. 

\begin{eqnarray}
E = \langle \partial_r X_t,\partial_rX_t \rangle &=& t^2p^2(\sinh^2(tpr)+1)= t^2p^2\cosh^2(tpr) \nonumber \\
F= \langle \partial_rX_t, \partial_\theta X_t \rangle &=& 0 \nonumber \\
G= \langle \partial_\theta X_t,\partial_\theta X_t \rangle &=& p^2\cosh^2(tpr). \nonumber
\end{eqnarray}

In order to obtain the second fundamental form we also compute the second derivative  and the normal vector to $X_t$.

\begin{eqnarray}
N_t &=& \frac{1}{\cosh^2(tpr)}(\cosh(tpr)\cos(p\theta),\cosh(tpr)\sin(p\theta),\sinh(tpr)\cosh(tpr)) \nonumber \\
\partial_r^2 X_t &=& t^2p^2 (\cosh(tpr)\cos(p\theta),\cosh(tpr)\sin(p\theta),0) \nonumber \\
\partial_\theta ^2 X_t &=&-p^2(\cosh(tpr)\cos(p\theta),\cosh(tpr)\sin(p\theta),0) \nonumber \\
\partial_{r\theta}^2 X_t &=& tp^2(-\sinh(tpr)\sin(p\theta),\sinh(tpr)\cos(p\theta),0). \nonumber
\end{eqnarray}
Thus 
\begin{eqnarray}
e= \langle \partial_r^2 X_t, N_t \rangle &=& t^2p^2 \nonumber \\
f= \langle \partial_{r\theta}^2 X_t, N_t \rangle &=&0 \nonumber \\
g= \langle \partial_\theta ^2 X_t, N_t \rangle &=&-p^2  \nonumber.
\end{eqnarray}
Now we find (see for example \cite{minimalsurfaces}) that 
$$K_t=\frac{eg-f^2}{EG-F^2}=-\frac{1}{\cosh^4(tpr)}$$
and 
\begin{eqnarray}
\Delta_{X_t}&=&\frac{1}{\sqrt{EG}}(\partial_r(E^{-1}\sqrt{EG}\partial_r)+\partial_\theta (G^{-1}\sqrt{EG}\partial_\theta)) \nonumber \\
&=&\frac{1}{tp^2\cosh^2(tpr)}(\frac{1}{t}\partial_r^2+t\partial_\theta^2).\nonumber
\end{eqnarray}

Thus we obtain that 
$$J_tu=\frac{1}{t^2p^2\cosh^2(tpr)}\left[\partial_r ^2 u +t^2(\partial_\theta ^2+\frac{2p^2}{\cosh^2(tpr)})\right]u.$$ 

We can easily see that $J_t: \mathcal{C}^{2,\alpha}_0 \rightarrow \mathcal{C}^{0,\alpha}$ is a symmetric operator on the space $L^2 ([-1,1]\times [-\pi,\pi],dA)$ with the usual inner product. Here $dA$ denotes the area element defined by $dA=t^2p^2\cosh^2(tpr)drd\theta)$. Thus $J_t$ is a Fredholm operator with null index and if the condition 1) of the Crandall-Rabinowitz theorem is realized so is condition 2). 
 
\item[\textbf{Step 2}:] Decomposition in Fourier series and existence of Jacobi fields. \\
We are looking for non trivial solutions of $J_tv=0$
\begin{equation}\label{jacobi1}
\left\{
\begin{array}{rcll}
\partial_r ^2 v +t^2(\partial_\theta ^2+\frac{2p^2}{\cosh^2(tpr)})v &= &0 \\
v(1,\theta)=v(-1,\theta)= 0
\end{array}
\right.
\end{equation}
Finding a solution of \eqref{jacobi1} is equivalent to find a solution of 

\begin{equation}\label{jacobi2}
\left\{
\begin{array}{rcll}
\partial_r ^2 v +\partial_\theta ^2v+\frac{2p^2}{\cosh^2(pr)}v &= &0 \\
v(t,\theta)=v(t,\theta)= 0
\end{array}
\right.
\end{equation}

\noindent(indeed if we have a solution $w(r,\theta)$ of \eqref{jacobi2} then we set $\tilde{r}=\frac{r}{t}$, $\tilde{r} \in [-1,1]$ and we obtain that     $v(\tilde{r},\theta)= w(t\tilde{r},\theta)$ is a solution of \eqref{jacobi1}).

Let us denote by $A_t:=[-t,t]\times [-\pi,\pi]$ and $A_\infty:=\R \times [-\pi,\pi]$.
We denote by $J$ the operator $\partial_r^2+\partial_\theta^2+ \frac{2p^2}{\cosh(pr)}$. With a slight abuse of language we will also call $J$ the Jacobi operator on the surface $X_t$. We consider the following eigenvalue problem:

\begin{equation}\label{eigenvalue1}
\left\{
\begin{array}{rclll}
\partial_r^2v+\partial_\theta^2v +(\frac{2p^2}{\cosh^2(pr)}+\lambda )v =0 \\
v(t,\theta)=v(-t,\theta)=0.
\end{array}
\right.
\end{equation}
Finding a solution of \eqref{jacobi2} is equivalent to say that $0$ is an eigenvalue of $J$ in $A_{t_0}$. Now we use a decomposition in Fourier series:
$$v(r,\theta)=\sum_{n\in \mathbb{Z}} v_n(r)e^{in\theta}. $$
Then $v$ is solution of \eqref{eigenvalue1} if and only if $v_n$ is solution of 
\begin{equation}\label{eigenvalueFourier}
\left\{
\begin{array}{rcll}
v_n''(r)+\big[\frac{2p^2}{\cosh^2(pr)}+(\lambda-n^2)\big]v_n =0 \\
v_n(t)=v_n(-t)=0
\end{array}
\right.
\end{equation}
for all $n\in \mathbb{Z}$.
Hence we consider another eigenvalue problem

\begin{equation}\label{eigenvalueradial}
\left\{
\begin{array}{rcll}
w''(r)+(\frac{2p^2}{\cosh^2(pr)}+ \mu)w &= &0 \\
w(t)=w(-t)=0
\end{array}
\right.
\end{equation}
The eigenvalue of \eqref{eigenvalue1} and \eqref{eigenvalueradial} are linked by the following relation 

\begin{equation}\label{relation}
\lambda =\mu + n^2.
\end{equation}

More precisely the preceding relation means that if $\lambda$ is an eigenvalue of problem \eqref{eigenvalue1} 
then there exists an integer $n$ and a real number $\mu$ which is an eigenvalue of \eqref{eigenvalueradial} such 
that $\lambda=\mu+n^2$ and conversely if $\mu$ is an eigenvalue of \eqref{eigenvalueradial} then $\mu+n^2$ is an 
eigenvalue of \eqref{eigenvalue1} for all integers $n$. Thus we deduce that the first eigenvalues of the two 
problems are the same $\lambda_1=\mu_1$ (with $n=0$). \\
We use a result of Barbosa-Do Carmo to say that if $t$ is small enough then the portion of catenoid $X_t$ is stable and $\lambda_1(A_t) \geq 0$. Indeed let us denote by $g_t$ the Gauss map of $X_t$. If $t$ is small enough then the area of $g_t(X_t)$ is strictly less than $2\pi$ (this area depends continuously of $t$ and when $t=0$ it is equal to $0$). Then using the result of Barbosa Do-Carmo \cite{BarbosaDoCarmo} we find that $X_t$ is stable for $t$ small enough and the first eigenvalue of the Jacobi operator is non negative.
 However we claim that the first eigenvalue of the infinite catenoid covered $p$-times is $-p^2$. This is because the function
$$w(r)=\frac{1}{\cosh(pr)}$$ is a positive function which satisfies
\begin{equation}\label{eigenvalueinfinite}
\left\{
\begin{array}{rcll}
w''(r)+(\frac{2p^2}{\cosh^2(pr)} -p^2)w &= &0 \\
w(+\infty)=w(-\infty)=0
\end{array}
\right.
\end{equation}
thus $w$ is the first eigenfunction of the operator $J$ and  $$\lambda_1(A_\infty)=\mu_1(A_
\infty)=-p^2.$$ We know use the continuity and strict monotonicity of the eigenvalues of the operator $J$
with respect to variations of the domain. This fact is due to the min-max principle of Courant-Fischer for eigenvalue of a selfadjoint operator and we refer to \cite{courant1989methods} p.407, \cite{Chavel} p.18 or the appendix A for a proof.  We obtain:
there exist $t_0,t_1,...t_{p-1}$ such that 
$$\mu_1(A_{t_0})=0,\ \mu_1(A_{t_1})=-1, \ \mu_1(A_{t_2})=-4, ..., \ \mu_1 (A_{t_{p-1}})=-(p-1)^2.$$

Thus there exist $p$ possible instants of bifurcations. More precisely there exist  $t_0,t_1,...t_{p-1}$ such that $0$ is an eigenvalue of the Jacobi operator $J_{t_k}$ on $A_{t_k}$. The corresponding eigenfunctions are sums of functions of the form $v_k(r)(A\cos(k\theta)+B\sin(k\theta))$, where $v_k$ are simple eigenfunctions of \eqref{eigenvalueradial} associated to the eigenvalue $-k^2$. 
The case $k=0$ is special. Indeed one can see that at $t=t_0$, $0$ is the first  eigenvalue of the Jacobi operator $J_{t_0}$, and one can show that bifurcation occurs at this instant but to give another catenoid. This is known as the stable/unstable catenoid bifurcation. For $t=t_k$, $k\neq0$, $k=1,...,(p-1)$ then $0$ is an eigenvalue of multiplicity at least 2 in the space $\mathcal{C}^{2,\alpha}_0 $. However we will prove that for $t=t_1$, $0$ is a simple eigenvalue of $J_t$ in the space $\mathcal{V}$.

\item[\textbf{Step 3:}] For $t=t_1$, $0$ is a simple eigenvalue in $\mathcal{V}$. \\

For $t=t_1$ because of the relation \eqref{relation}, we are able to rank the eigenvalues of $J_{t_1}$, denoted by $\lambda_1<\lambda_2\leq...$. By definition of $t_1$ we have 
$$\lambda_1(A_{t_1})=\mu_1(A_{t_1})=-1.$$
Denoting the eigenvalues of \eqref{eigenvalueradial} by $\mu_1 < \mu_2 \leq...$  we claim that it holds that $\mu_2(A_{t_1})>0$. The proof of this fact is the same of the discussion in the paper of Shiffman \cite{Shiffman} p.82. We reproduce this proof for the comfort of the reader. A direct computation shows that solutions of the problem \eqref{eigenvalueradial} with $\mu=0$ are of the form 
\begin{equation}\label{26}
w(r)=a \tanh(pr)+b[1-(pr)\tanh(pr)], \ \text{for} \ a,b \in \R.
\end{equation}

For $b \neq 0$, the zeros of the function \eqref{26} occur at the solutions of 
$$\coth(pr)=pr-c, \ \text{where} \ c=\frac{-a}{b}. $$
We can see that there are two solutions of this equation $r_1(c),r_2(c)$ with $r_1(c)<0<r_2(c)$; and both $r_1(c)$ and $r_2(c)$ are increasing functions of $c$ with $r_1(c)\rightarrow -\infty$, $r_2(c)\rightarrow 0$ as $c \rightarrow -\infty$ and $r_1(c) \rightarrow 0, r_2(c) \rightarrow +\infty$ as $c \rightarrow +\infty$. Thus every real number is covered exactly once by all the numbers $r_1(c),r_2(c)$ the number $0$ being covered by the case $b=0$. Select one function \eqref{26} which vanishes at $r_1$. Then, in the interval $r_1 \leq r \leq r_2,$ this function vanishes for at most one additional value of $r$. This means in accordance with Sturm-Liouville theory, that 
$$\mu_2(A_{t_1})>0.$$
Another way to see this fact is the following: since the functions \eqref{26} have only two zeros, the Courant nodal's theorem allows us to say that $0$ can not be an $n$th-eigenvalue of \eqref{eigenvalueradial} with $n\geq 2$. For $t$ small we have $\mu_2(A_t) >\mu_1(A_t)>0$, hence because the eigenvalues $\mu_n(A_t)$ are decreasing with respect to $t$ (see Appendix A) we find that $\mu_2(A_t)>0$.
Thus we can deduce that 

$$\lambda_2=\mu_1+1=0< \mu_2$$ and we have the alternative
$$\lambda_3=-1+4, \ \lambda_3=\mu_2+1 \ \text{or} \ \lambda_3=\mu_3.$$
In any case we have $\lambda_3>0$ and for $i\geq 3$, $\lambda_i >0$. 
Hence the eigenfunctions associated to $0$ are only of the form $A_1v_1\cos(\theta)+B_1v_1\sin(\theta)$ with 
$v_1$  the first eigenfunction of $J_{t_1}$ and $A_1,B_1$ two constants. However because we choose to work in the 
space $\mathcal{V}=V \cap \mathcal{E}$, the only eigenfunctions admissible in $\mathcal{E}$ are of the form 
$A_1v_1\cos(\theta)$. Thus if $0$ is an eigenvalue of $J_{t_1}$ in $\mathcal{E}$ it is simple. We denote by 
$u_1:=\frac{v_1\cos(\theta)}{\|v_1\cos(\theta)\|_{L^2}}$ we prove in the next step that $u_1$ is indeed in $\E$ and hence $0$ is a simple 
eigenvalue of $J_{t_1}$ in $\mathcal{E}$. 
\\

\item[\textbf{Step 4:}] $u_1(-r,\theta)=u_1(r,\theta)$ for all $(r,\theta) \in [-1,1] \times [-\pi,\pi]$.

This comes from the fact that $u_1(r,\theta)=\frac{v_1(r)\cos(\theta)}{\|v_1\cos(\theta)\|_{L^2}}$, with $v_1$ which is the first eigenfunction 
associated to the eigenvalue $-1$ of the problem \eqref{eigenvalueradial}. Indeed if we set $\tilde{v_1}(r)=v_1(-
r)$ for $r \in [-1,1]$. Then $\tilde{v_1}$ is also a solution of \eqref{eigenvalueradial} with $\mu=-1$. Thus we 
deduce that $\tilde{v_1}=av_1$ for some real constant $a$. Because $\|\tilde{v_1}\|_{L^2}=\|v_1\|_{L^2}$ we find 
that $\tilde{v_1}(r)=\pm v_1(r)$. However if $\tilde{v_1}(r)=-v_1(r)$ this would imply that $v_1$ is odd which is 
not possible because $v_1 >0$ in $[-1,1]$ (this is the first eigenfunction). Hence we find that $\tilde{v_1}=v_1$ 
and then that $v_1$ is symmetric with respect to $r\mapsto -r$. This proves the claim. This also proves the fact that for $t=t_1$ the Jacobi operator $J_{t_1}$ is not invertible in the space $\E$. 

\begin{remark} We are not able to prove that $0$ is a simple eigenvalue of $J_{t_k}$ in $\mathcal{E}$ for $k 
\geq2$. Indeed for example if $k=2$, one has that $\mu_1(A_{t_2})=-4$, but if it occurs that $\mu_2(A_{t_2})=-1$ 
then the eigenfunctions associated to $0$ are of the form $v_2(r)(A\cos(2\theta)+B\sin(2\theta))$ for some constants $A$ and 
$B$ and $v_2$ the first eigenfunction of $X_{t_2}$, or of the form $f(r)[A\cos(\theta)+B\sin(\theta)]$ with $f$ 
an eigenfunction of \eqref{eigenvalueradial} associated to the eigenvalue $-1$.
\end{remark}

\item[\textbf{Step 5:}] Verification of the Crandall-Rabinowitz condition.

Here we assume that $t=t_1$. In order to apply the Crandall-Rabinowitz theorem we need to verify one last condition. We still denote by $u_1=\frac{v_1\cos(\theta)}{\|v_1\cos(\theta)\|_{L^2}}$ the eigenfunction associated to $0$ in the space $\E$ and we set $M:=D_{u,t}F(t_1,0)$, $L:=J_{t_1}$. $M$ is equal to the operator obtained by differentiating $J_t$ with respect to $t$ and taking $t=t_1$.  We must check that 
\begin{equation}\label{crandall}
Mu_1 \notin \text{Im}(L).
\end{equation}
In order to prove this it suffices to prove that $\langle Mu_1,u_1 \rangle \neq 0$ where $\langle \cdot,\cdot\rangle$ stands for the inner product in $L^2$. Indeed if $Mu_1$ is in $\text{Im}(L)$ then we can write $Mu_1=Lw$ for some $w$ in $\E$ but then
\begin{eqnarray}
\langle Mu_1,u_1 \rangle &=& \langle Lw,u_1 \rangle \nonumber \\
&=& \langle w,L^*u_1 \rangle ,\ \text{with $L^*$ the adjoint of $L$} \nonumber \\
&=&\langle w,Lu_1 \rangle, \ \text{because $L$ is symmetric} \nonumber \\
&=&0 ,\ \text{because $Lu_1=0$ by definition}. \nonumber
\end{eqnarray}

Thus we must prove that $ \langle Mu_1,u_1 \rangle \neq 0$. We use an argument taken from \cite{ABM} and 
\cite{Xavier}.
We already mentioned that the first eigenvalue of the catenoid is strictly decreasing with respect to variations of the domain, i.e.
$\lambda_1(A_t)=\mu_1(A_t)$ is strictly decreasing with respect to $t$. Furthermore this quantity depends 
smoothly on $t$ and if  we denote by $u_t$ the eigenfunction in $\E$ associated to the eigenvalue $\lambda_2(A_t)=\lambda_1(A_t)+1$ (c.f.relation \eqref{relation}) the function $u_t$ depends smoothly on $t$. This is due to the fact that we can rewrite the eigenvalue problem \eqref{jacobi1} on a fix domain with a new operator which depends analytically on $t$. We then use the \textit{Kato selection theorem} (see \cite{Kato} and the appendix A) to obtain the fact that the eigenvalue and  the eigenfunction are analytic. We have  $u_{t_1}=u_1$ and

$$\lambda_2(A_t)= \langle J_t u_t ,u_t \rangle $$ thus

\begin{eqnarray}
0 &> &\lambda_2'(A_{t_1})= \frac{d}{dt}_{|t=t_1} \langle J_t u_t ,u_t \rangle \nonumber \\
&=&\langle[\frac{d}{dt}_{|t=t_1}J_t ]u_1,u_1\rangle + \langle J_{t_1}\frac{d}{dt}_{|t=t_1}u_t,u_1\rangle +\langle J_{t_1}u_1,\frac{d}{dt}_{|t=t_1}u_t\rangle \nonumber \\
&=& \langle[\frac{d}{dt}_{|t=t_1}J_t ]u_1,u_1 \rangle, \ \text{because $J_{t_1}$ is symmetric and $J_{t_1}u_1=0$} \nonumber
\end{eqnarray} 
\noindent but $[\frac{d}{dt}_{|t=t_1}J_t] u_1=Mu_1$ by definition. Hence we proved that $\langle Mu_1,u_1 \rangle <0 $ and then \eqref{crandall} is satisfied. 
\end{itemize}

We have all the ingredients to apply the bifurcation theorem of \cite{Crandall} and we obtain the following:

\begin{theorem}
For any integer $p\geq 2$ there exists an instant $t^*(p)$ such that $(t^*,0)$ is a bifurcation point for $F$ (c.f. \eqref{bifurcation}). In addition the set of solutions of $F=0$  near $(t^*,0)$ is formed by two $\mathcal{C}^1$ Cartesian curves which intersect each other transversely in $0$ . 
\end{theorem}

\begin{proof}
It suffices to take $t^*=t_1(p)$ in the previous construction and apply the bifurcation theorem.
\end{proof}
\begin{itemize}

\item[\textbf{ Step 6:}] The bifurcating surfaces are not catenoids.

Now in order to conclude the proof of theorem \ref{nonsymmetricmminsurf} we must check that 
parametrizations $Y_t=X_t+u(t)N_t$, with $u(t)\neq 0$ the solution of \eqref{bifurcation} obtained via the
bifurcation, are not portions of catenoids and that there are symmetric with respect to the planes $P=\{z=0\}$ 
and $P'=\{x=0\}$. The surfaces $Y_t$ are indeed symmetric with respect to these planes because we choose to 
work  in the space $\E$ to prove the bifurcation. Thus the solution $u$ of \eqref{bifurcation} inherits the 
symmetry 
of the space $\E$ and this implies the symmetries of the surfaces $Y_t$. Indeed because of the symmetry of $X_t$ and $N_t$ if $u \in \E$ then $Y_t=X_t+u(t)N_t$ satisfies that its coordinates $(y_1,y_2,y_3)$ are such that
$$(y_1(r,-\theta),y_2(r,-\theta),y_3(r,-\theta))=(y_1(r,\theta),-y_2(r,\theta),y_3(r,\theta)).$$
 
 We prove now that $Y_t$ are not catenoids. Recall that, because $u(t) \in \mathcal{C}^{2,\alpha}
 _0$ (i.e. $u(t)$ is zero on the boundaries of the domain) $X_t$ and $Y_t$ have same boundaries. We know that 
 there are at most two catenoids which pass through two 
 given circles  in parallel planes. One stable catenoid and one unstable catenoid. $X_t$ 
 for $t$ near $t_1$ is the unstable catenoid because $\lambda_1(A_t) <0$. If $t$ is near
enough from $t_1$ then $u(t)$ is small (that is $\|u\|_{\mathcal{C}^{2,\alpha}}$ is small).
Thus $Y_t=X_t+u(t)N_t$ is close to the unstable catenoid, but the stable catenoid  is ``far
away" from the unstable one. More precisely for $t$ near $t_1$, $X_t+u(t)N_t$ is in a
 small tubular neighborhood of $X_t$ whereas the stable catenoid is not contained in
such small tubular neighborhood. Note that when $t=t_0$ the stable and unstable
catenoid coincide and for $t$ near $t_0$ the stable catenoid is in a small tubular
 neighborhood of the unstable one, this is not the case here because $t_1>t_0$.
   %there exists a constant $B>0$ such that there exist a point $M_t$ on the stable catenoid such that $\dist(M_t,\mathbf{C_t})>0$. Here $\mathbf{C_t}$ denotes the unstable catenoid parametrized by $X_t$ near $t_1$. Hence if $u$ is small enough (that is if $t$ is chosen sufficiently close from $t_1$) then $Y_t+u(t)N_t$ can not be the stable catenoid passing through the same circles than $X_t$.%  
We also claim that $Y_t$ can not be a re-parametrization of the unstable catenoid parametrized by $X_t$ if $u$ is small enough but different from $0$. Besides for $t-t_1$ small the surfaces $Y_t=X_t+u(t)N_t$ are on both sides of the catenoid. Indeed making an expansion of $Y_t$ for $t-t_1$ small we find that:

\begin{equation}\nonumber
Y_t=X_{t_1}+(t-t_1)[\partial_t X_{|t=t_1} + u(t_1) \partial_tN_{|t=t_1}+ \partial_tu_{|t=t_1} N_{t_1}]+o(t-t_1)
\end{equation}

$\partial_tu_{|t=t_1}=u_1$ is the Jacobi field found previously thus we know that $u_1=\frac{v_1(r)\cos(\theta)}{\|v_1(r)\cos(\theta)\|_{L^2}}$. We can check that $\partial_t X_{|t=t_1}$ and $\partial_tN_{|t=t_1}$ are both tangent to the surface $X_{t_1}$. Hence we can focus on the normal variation. Because the sign of function $u_1$ changes with $\theta$ we obtain  that $Y_t=X_t+u(t)N_t$ is on both sides of the catenoid for $t$ near $t_1$. Thus we can deduce theorem \ref{nonsymmetricmminsurf}. 
\end{itemize}
\end{proof}
From this theorem and the equivalence between finding minimal surfaces bounded by two coaxial circles in parallel planes and finding solutions of \eqref{semi-stiff 1} with $c<0$ we obtain:

\begin{theorem}
There exist non-radial solutions of \eqref{semi-stiff 1}.
\end{theorem}

\begin{figure}[ht!]

\begin{center}
\includegraphics[scale=1.5]{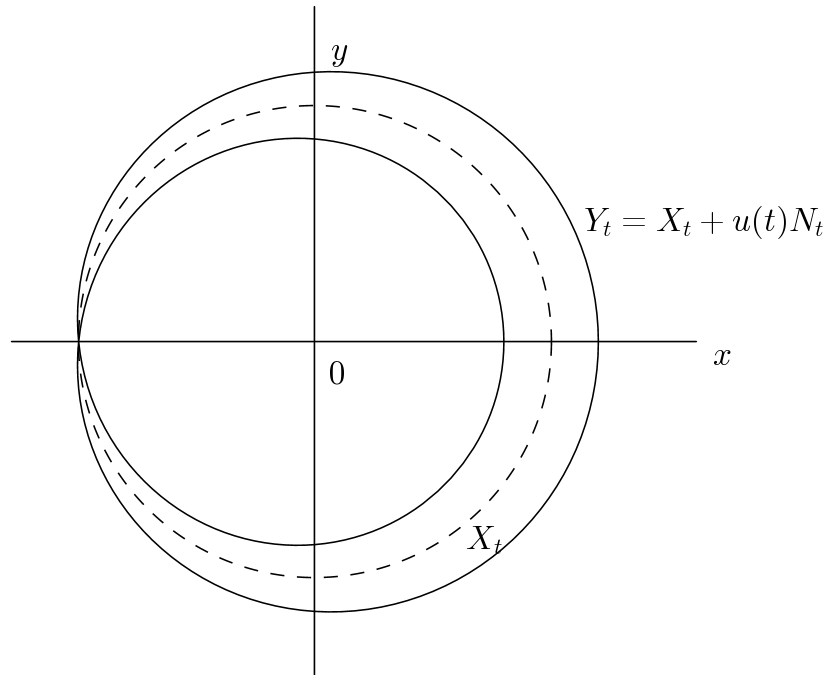}
\end{center}
\caption{View of bifurcating solutions for $p=2$ in the plane $\{z=0\}$.}
\end{figure}

\begin{figure}[ht!]
\begin{center}
\scalebox{0.6}{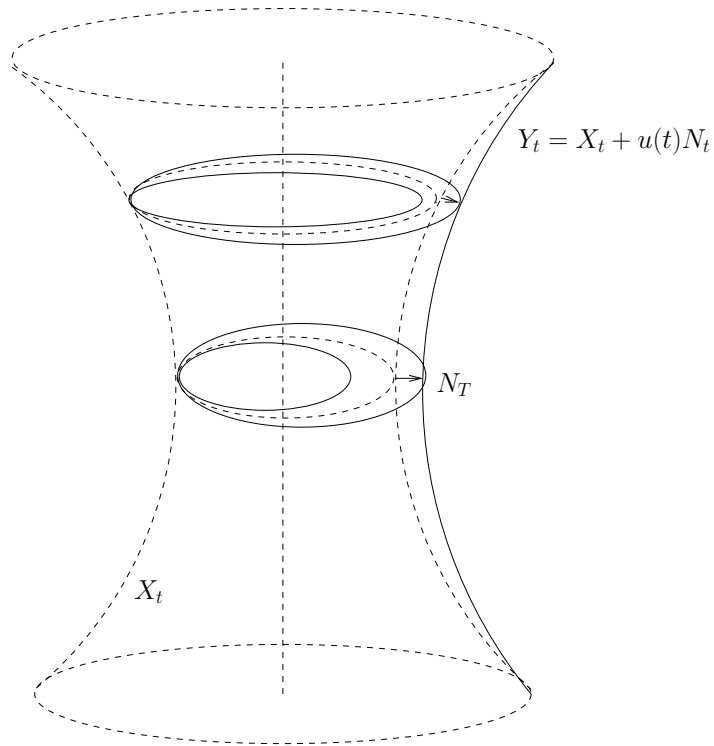}
\end{center}
\caption{Bifurcating solutions for $p=2$.}
\end{figure}

\section*{Conclusion and open problems}

We classified all holomorphic (or anti-holomorphic) solutions of \eqref{semi-stiff 1} and we discussed some properties of radial solutions of \eqref{semi-stiff 1}. We also discussed the link between problem \eqref{semi-stiff 1} and minimal surfaces. Besides thanks to this link we built new minimal surfaces in $\R^3$ bounded by two circles in parallel planes and then new non radial solutions of \eqref{semi-stiff 1}. Finding solutions to problem \eqref{semi-stiff 1} is a non trivial issue because of the lack of compactness of the problem. Studying problem \eqref{semi-stiff 1} may be useful to understand more general problems of lack of compactness like the G.L equations with semi-stiff boundary conditions.

Three main questions remain open after the completion of this work.

\textbf{Open problem 1:} Do the non radial solutions constructed in section 6 minimize the energy in the class $\I_{p,p}$? If the minimal surfaces obtained from these solutions by proposition \ref{buildingminsurf1} were not symmetric with respect to some horizontal planes then we would be able to deduce that the solutions are not minimizers of $E$ in $\I_{p,p}$ because of proposition \ref{planesymmetry}. But we constructed minimal surfaces which are symmetric with respect to some horizontal planes hence the question remains open.\\

If the non radial solutions of \eqref{semi-stiff 1} do not minimize $E$ in $\I_{p,p}$, do there exist minimizer of $E$ in this class? Do we have $\varrho_p=\varrho'_p$ in theorem \ref{theorem2}? If this is not the case it would imply that there exists non radially symmetric minimizing solutions which have the same energy as the minimizing radial solution.
These questions are interesting from the point of view of the study of lack of compactness of the problem. \\

\textbf{Open problem 2:} Can we construct (non-minimizing) non-radial solutions of \eqref{semi-stiff 1} with $c>0$? If yes what is the geometry of the minimal surfaces linked to these new solutions? \\

An other open question is:

\textbf{Open problem 3:} Do there exist non constant solution of \eqref{semi-stiff 1} in $\I_{0,0}$? \\

We can not give an answer to that seemingly simple question, however we can prove that if such non constant solution in $\I_{0,0,}$ $u$ exists then its Hopf differential satisfies $\H_u(z)= \frac{c}{z^2}$ with $c>0$ and then it has a geometry of ``helicoid" type according to the terminology we used in section 5. 

\begin{proposition}
Let $u\in \I_{0,0}$ be a non constant solution of \eqref{semi-stiff 1} then $\H_u(z)= \frac{c}{z^2}$ with $c>0$.
\end{proposition}

\begin{proof}
Let $u\in \I_{0,0}$ be a non constant solution of \eqref{semi-stiff 1}, we know from lemma \ref{Hopfcst} that 
$$\H_u(z)=\frac{c}{z^2}, \ \text{in} \ A $$
with $c$ a real constant. If $c=0$ then $u$ is holomorphic or antiholomorphic and then $u$ must belong to $\I_{r,s}$ with $r>0>s$ or with $r<0<s$ because of lemma \ref{obstruction} which is a contradiction. If $c<0$ then we obtain that 
$$r^2|\partial_ru|^2-|\partial_\theta u|^2 =c <0, \ \text{in} A .$$

Thus $\partial_\theta u$ does not vanish in $A$. But on $\partial A$ because $u$ is smooth and $|u|=1$ on can write $u =|u|e^{i\varphi}$ the fact that $|u|=1$ on $\partial A$ gives us that $|\partial_\theta u|=|\partial_\theta \varphi|$ and then $\partial_\theta \varphi$ does not vanish on $\partial A$ so has constant sign on $\mathbb{S}^1$ and on $C_\varrho$. Let us assume that $\partial_\theta \varphi >0$. Then we  find that 
$$\deg(u,\mathbb{S}^1)=\frac{1}{2\pi}\int_{\mathbb{S}^1} u \wedge u_\tau =\frac{1}{2\pi}\int_0^{2\pi} \partial_\theta \varphi d\theta > 0$$
which is a contradiction with the fact that $u\in \I_{0,0}$ then we must have $c>0$.

\end{proof}

\section*{Appendix A: On the properties of the eigenvalues of the Jacobi operator }

The aim of this appendix is to collect and prove some results about the monotonicity and differentiability of eigenvalues of some elliptic operators. These results were used in a crucial way in the proof of theorem \eqref{nonsymmetricmminsurf}. We refer to \cite{Chavel}, \cite{courant1989methods} and \cite{Kato} for the proofs of these results. Like before we let $A_t:= [-t,t] \times [-\pi,\pi]$ for $t>0$. We consider the following eigenvalue problems

\begin{equation}\label{JacobiA1}
\left\{
\begin{array}{rcll}
\partial_r ^2 v +\partial_\theta ^2 v+\frac{2p^2}{\cosh^2(pr)}v +\lambda &= &0 \\
v(1,\theta)=v(-1,\theta)= 0
\end{array}
\right.
\end{equation}
and 

\begin{equation}\label{eigenvalueradialA2}
\left\{
\begin{array}{rcll}
w''(r)+(\frac{2p^2}{\cosh^2(pr)}+ \mu)w &= &0 \\
w(t)=w(-t)=0
\end{array}
\right.
\end{equation}

For $t>0$ we denote by $\lambda_1(A_t) \leq \lambda_2(A_t) \leq...$ the eigenvalues of problem \eqref{JacobiA1} and by $\mu_1(A_t) \leq \mu_2(A_t) \leq ...$ the eigenvalues of problem \eqref{eigenvalueradialA2}. We already pointed out that the eigenvalues of problem \eqref{JacobiA1} and \eqref{eigenvalueradialA2} are linked. If $\lambda$ is an eigenvalue of \eqref{JacobiA1} then there exist an integer $n$ and an eigenvalue $\mu$ of \eqref{eigenvalueradialA2} such that 
$$\lambda= \mu +n^2.$$
Conversely  if $\mu$ is an eigenvalue of \eqref{eigenvalueradialA2} then $\lambda=\mu+n^2$ is an eigenvalue of \eqref{JacobiA1} for all $n\in \mathbb{Z}$. From this relation we easily deduce that 
$$\lambda_1(A_t)=\mu_1(A_t).$$
The first result we used is the strict monotonicity of the eigenvalues $\lambda$ and $\mu$ with respect to domain variations:

\begin{proposition}
For all $k$ in $\mathbb{N}$ the functions $t \mapsto \lambda_k(A_t)$ and $t \mapsto \mu_k(A_t)$ are decreasing on $]0,+\infty[$.
\end{proposition}

\begin{proof}
Let $k \in \mathbb{N}$, we only do the proof for $t \mapsto \lambda_k(t)$ because it is similar for $t\mapsto \mu_k(t)$.
For the first eigenvalue $\lambda_1(A_t)$ using the Rayleigh quotient we have that 
$$\lambda_1(A_t)=\min_{v\in H^1_0(A_t)} \{\langle Jv,v \rangle ; \|v\|_{L^2}=1\}.$$
Let $t<t'$, we denote by $v_1(A_t)$ the first eigenfunction associated to the eigenvalue $\lambda_1(A_t)$. We can extend $v_1$ in a function $\tilde{v_1}$ defined on $A_{t'}$ by setting 
$$\tilde{v_1}(x)= \begin{cases} v_1(x), \ \text{if} \ x\in A_t \\
0, \ \ \text{if} \ x \in A_{t'}\setminus A_t.
\end{cases}$$

We then find that 
\begin{eqnarray}
\lambda_1(A_{t'}) & = & \min _{v\in H^1_0(A_{t'})} \{\langle Jv,v \rangle ; \|v\|_{L^2}=1 \} \nonumber \\
 & \leq & \langle J\tilde{v_1}, \tilde{v_1} \rangle \nonumber \\
 & \leq  &\lambda_1(A_t). \nonumber
\end{eqnarray}

Besides if $\lambda_1(A_t)=\lambda_1(A_{t'})$ it is easy to see that $\tilde{v_1}$ is the first eigenvalue of $J$ on $A_{t'}$. But a first eigenvalue is positive in the interior if the domain and $\tilde{v_1}$ vanishes in the interior of $A_{t'}$ this is a contradiction. \\
For the other eigenvalues one can use the Min-Max theorem (cf. \cite{courant1989methods}) to obtain that 
$$\lambda_k(A_t)=\min \{ \max_{v \in E_k, \|v\|_{L^2}=1} \langle Jv,v \rangle, \ E_k \subset H^1_0(A_t) \ \text{subspace of dimension} \ k \}.$$
Because the subspaces of dimension $k$ of $H^1_0(A_t)$ can be viewed as subspaces of dimension $k$ of $H^1_0(A_{t'})$ (just by extending a function in the initial subspace by $0$ in $A_{t'} \setminus A_t$) we find that 
$$\lambda_k(A_{t'}) \leq \lambda_k(A_t).$$
Furthermore if there is equality one can see that an eigenfunction $v_k$ of $J$ in $A_{t'}$ vanishes in  $A_{t'} \setminus A_t$ and this is a contradiction.
\end{proof}

In the last section of the paper, during the proof of theorem \ref{nonsymmetricmminsurf} we used the fact that the function $t \mapsto \lambda_1(A_t)$ is differentiable for $t>0$ and that the first eigenfunction associate to this eigenvalue is also differentiable with respect to $t$. This fact result from the following 

\begin{proposition}\label{analytic}
Let $t_0>0$ and $\lambda$ be an eigenvalue of multiplicity $k$ of $J_{t_0}$ ($J_{t_0}$ denotes the Jacobi operator for functions defined on $A_{t_0}$). Then there exists $k$ real numbers $\Lambda^i$ and $k$ functions $\Phi^i$, $i=1,...,k$ which are real analytic functions of $t$ such that 
\begin{itemize}
\item[1)] $J_t \Phi^i_t +{\Lambda^i}_t \Phi^i_t =0$
\item[2)] $\Lambda^i({t_0})=\lambda$
\item[3)] $\{ \Phi^i_t \}_{i=1,...,k}$ is orthonormal for all $t$.
\end{itemize}
\end{proposition}

\begin{proof}
We first note that if $\lambda(A_t)$ is an eigenvalue for the problem \eqref{JacobiA1} on $A_t$ then $\tilde{\lambda}(t):=t^2\lambda (A_t)$ is an eigenvalue of the following problem:

\begin{equation}\label{jacobi1'}
\left\{
\begin{array}{rcll}
\partial_r ^2 v +t^2(\partial_\theta ^2+\frac{2p^2}{\cosh^2(ptr)})v +\tilde{\lambda}(t)&= &0 \\
v(1,\theta)=v(-1,\theta) &=& 0
\end{array}
\right.
\end{equation}
Thus it suffices to prove the proposition \ref{analytic} for an eigenvalue $\tilde{\lambda}$ of \eqref{jacobi1'}. We used an argument taken from \cite{MarcelBerger}. For $z$ complex we can define an operator $J_{z}: \mathcal{C}^\infty(A_1,\C) \rightarrow \mathcal{C}^\infty(A_1,\C)$ defined by
$$J_zv:= \partial_r ^2 v +z^2(\partial_\theta ^2+\frac{2p^2}{\cosh^2(pzr)})v.$$

The domain of this operator is the space $H^2(A_1,\C)$ and it is independent of $z$. Furthermore for $v \in H^2(A_1,\C)$ the function $z \mapsto J_z f$ is holomorphic for $z$ in a neighborhood of $t_0>0$ because the application $z\mapsto z^2$ and $z \mapsto \frac{1}{\cosh(pzr)}$ are holomorphic on   $\C \setminus \{ \frac{i\pi}{pr}(l+\frac{1}{2}), l \in \mathbb{Z} \}$ for $r>0$. We can then deduce that $J_z$ is a holomorphic family of operator of type (A) in the sense of \cite{Kato} p.375, for $z$ near $t_0>0$. We can easily see that his family is also selfadjoint on $L^2(A_1)$. We can thus apply the result of the last paragraph of p.386 in \cite{Kato} and this proves the proposition.
\end{proof}

\section*{Appendix B: On the harmonic extension in a circular domain}

In this appendix we give the computations which lead to a formula for the harmonic extension of $g=e^{ip\theta}$ on $\mathbb{S}^1$ and $g=\alpha e^{iq\theta}$ for $\alpha \in \mathbb{S}^1$ (see also appendix D in \cite{unpublished}). These formulas where needed in the proof of proposition \ref{linkwithSteklov}. 

For any function $g \in H^{\frac{1}{2}}(\partial A)$, using Fourier series one can write 
$$g=\sum_{n\in \mathbb{Z}} a_ne^{in\theta}, \ \text{on} \ C_\varrho $$
$$g=\sum_{n\in \mathbb{Z}} b_ne^{in\theta}, \ \text{on} \ \mathbb{S}^1,$$

\noindent with $\sum_{n\in \mathbb{Z}}|n||a_n|^2 < +\infty$ and $\sum_{n\in \mathbb{Z}}|n||b_n|^2 < +\infty$.
Let $u$ be the harmonic extension of $g$ to $A$. We may write in polar coordinates
\begin{equation}\nonumber
u=A_0+B_0\ln r +\sum_{n\neq 0} (A_nr^{\n}+B_nr^{-\n})e^{in\theta}.
\end{equation}

\noindent Identification of the coefficients on $\partial A$ yields 

\begin{equation}\nonumber
A_0=b_0, \ \ \ B_0=\frac{a_0-b_0}{\ln \varrho}
\end{equation}

\noindent and 
\begin{equation}\nonumber
A_n=\frac{b_n-a_n\varrho^{\n}}{1-\varrho^{2|n|}}, \ \ \ B_n=\frac{\varrho^{\n}(a_n-\varrho^{\n} b_n)}{1-\varrho^{2|n|}}
\end{equation}

We now want to obtain a formula for $u$ the harmonic extension of $g=e^{ip\theta}$ on $\mathbb{S}^1$ and $g=\alpha e^{iq\theta}$ on $C_\varrho$ (with $(p,q)\in \mathbb{Z}^2$ and $\alpha \in \mathbb{S}^1)$. Application of the previous formulas yields if $p\neq q$

\begin{equation}\nonumber
A_p=\frac{1}{1-\varrho^{2|p|}}, \ \ \ A_q=\frac{-\alpha \varrho^{\q}}{1-\varrho^{2|q|}}, \ \ A_n=0 \ \text{if} \ n\neq p,q
\end{equation}

\begin{equation}\nonumber
B_p=\frac{-\varrho^{2\p}}{1-\varrho^{2|p|}}, \ \ \ B_q=\frac{\alpha \varrho^{\q}}{1-\varrho^{2|q|}}, \ \ A_n=0 \ \text{if} \ n \neq p,q 
\end{equation}

\noindent and if $p=q$ one obtains

\begin{equation}\nonumber
A_p=\frac{1-\alpha \varrho^{\p}}{1-\varrho^{2|p|}}, \ \ \ B_p=\frac{\varrho^{\p} (\alpha-\varrho^{\p})}{1-\varrho^{2|p|}}, \ A_n=0 \ \text{if} \ n\neq p.
\end{equation}

In any case we obtain

\begin{equation}
u=\frac{1}{1-\varrho^{2|p|}}(r^{\p}-\frac{\varrho^{2\p}}{r^{\p}})e^{ip\theta}-\frac{\alpha \varrho^{\q}}{1-\varrho^{2|q|}}(r^{\q}-\frac{1}{r^{\q}})e^{iq\theta}.
\end{equation}
We can thus deduce that 

\begin{equation}
\partial_r u =\frac{|p|}{1-\varrho^{2|p|}}(r^{\p-1}+\frac{\varrho^{2\p}}{r^{\p+1}})e^{ip\theta}-\frac{\q\alpha \varrho^{\q}}{1-\varrho^{2|q|}}(r^{|q|-1}+\frac{1}{r^{\q+1}})e^{iq\theta}.
\end{equation}

This yields on $\mathbb{S}^1$, with $\alpha=e^{i\varphi} $
\begin{eqnarray}
u\wedge \partial_r u &= &\frac{|q|\varrho^{|q|}}{(1-\varrho^{2|p|})(1-\varrho^{2|q|})}(r^{\p}-\frac{\varrho^{2\p}}{r^{\p}})(r^{\q-1} 
+\frac{1}{r^{\q+1}})\sin([p-q]\theta+\varphi) \nonumber \\ &+&\frac{\p \varrho^{\q}}{(1-\varrho^{2|p|})(1-\varrho^{2\q})}(r^{\q}-\frac{1}{r^{\q}})(r^{\p-1}+\frac{\varrho^{2\p}}{r^{\p+1}})\sin([p-q]\theta+ \varphi)
\end{eqnarray}

With a similar computation on $C_\varrho$ one can see that $u\wedge \partial_\nu u =0$ on $\partial A$ if and only if $q=p$ and $\alpha=1$ or $q=p$ and $\alpha =-1$. This concludes the proof of proposition \ref{linkwithSteklov}.
\section*{Acknowledgements}

The second author wishes to warmly thank Etienne Sandier for his help and for constant support during the elaboration of this paper. He also wants to thank Mickaël Dos Santos, Yuxin Ge, Laurent Mazet, Rabah Souam, Eric Toubiana and his colleague Peng Zhang for very useful discussions on this paper.
\nocite*

\bibliographystyle{plain}
\bibliography{lapla}

\end{document}

%% file: bifurcation3.pdf_t
\begin{picture}(0,0)%
\includegraphics{bifurcation3.pdf}%
\end{picture}%
\setlength{\unitlength}{4144sp}%
\begingroup\makeatletter\ifx\SetFigFont\undefined%
\gdef\SetFigFont#1#2#3#4#5{%
  \reset@font\fontsize{#1}{#2pt}%
  \fontfamily{#3}\fontseries{#4}\fontshape{#5}%
  \selectfont}%
\fi\endgroup%
\begin{picture}(5416,7350)(4493,-7539)
\put(8731,-4021){\makebox(0,0)[lb]{\smash{{\SetFigFont{17}{20.4}{\familydefault}{\mddefault}{\updefault}{\color[rgb]{0,0,0}$N_T$}%
}}}}
\put(9541,-1591){\makebox(0,0)[lb]{\smash{{\SetFigFont{17}{20.4}{\familydefault}{\mddefault}{\updefault}{\color[rgb]{0,0,0}$Y_t=X_t+u(t)N_t$}%
}}}}
\put(5716,-6091){\makebox(0,0)[lb]{\smash{{\SetFigFont{17}{20.4}{\familydefault}{\mddefault}{\updefault}{\color[rgb]{0,0,0}$X_t$}%
}}}}
\end{picture}%